\documentclass{article} 

\usepackage[utf8]{inputenc}
\usepackage[T1]{fontenc}
\usepackage{lmodern}
\usepackage{amsmath, amssymb, amsthm, amstext,mathrsfs}
\usepackage{microtype}
\usepackage[english=usenglishmax]{hyphsubst}
\usepackage{esint}
\usepackage{booktabs}
\usepackage{graphicx}
\usepackage[babel,german=quotes]{csquotes}
\usepackage{bibentry}
\usepackage{tabularx}		
\usepackage{commath}			% for \norm{...}
\usepackage{setspace}
\usepackage{bm}
\usepackage{threeparttable}
\usepackage{comment}\includecomment{notthisone}
\usepackage{enumerate}
\usepackage{epstopdf}
\usepackage{enumerate}
\usepackage{dsfont}
\usepackage{pstool}
\usepackage{subfig}
\usepackage{natbib}
\usepackage{hyperref}
\usepackage{authblk}
\usepackage{algorithm}
\usepackage{algcompatible}
\usepackage{algorithmicx}
\usepackage{array, booktabs}
\usepackage{setspace}
\usepackage{mathcomp}
\usepackage{multirow}
\usepackage[paper=a4paper,left=25mm,right=25mm,top=25mm,bottom=25mm]{geometry}

\floatname{algorithm}{Algorithm}

\providecommand{\keywords}[1]{\textbf{\textit{Keywords: }} #1}
\providecommand{\classification}[1]{\textit{2020 Mathematics Subject Classification: } #1}
\makeatletter
\def\BState{\State\hskip-\ALG@thistlm}
\makeatother

\usepackage{mathtools}
\mathtoolsset{showonlyrefs}

\makeatletter

\newtheoremstyle{Definition}
  {0.2cm}                   %Space above
  {0.2cm}                   %Space below
  {\normalfont}           %Body font
  {}                      %Indent amount (empty = no indent,
  {\bfseries}  						%Thm head font
  {.}                     %Punctuation after thm head
  { }              				%Space after thm head: " " = normal interword
  {}
\newtheoremstyle{Theorem}
  {0.2cm}                   %Space above
  {0.2cm}                   %Space below
  {\itshape}           		%Body font
  {}                      %Indent amount (empty = no indent,
  {\bfseries}  						%Thm head font
  {.}                     %Punctuation after thm head
  { }              				%Space after thm head: " " = normal interword
  {}
\theoremstyle{Theorem}
	\newtheorem{cor}{Corollary}
	\newtheorem{prop}{Proposition}
	
	\newtheorem{thm}{Theorem}

	\newtheorem{ass}{Assumption}

\theoremstyle{Definition}

\newcommand{\Levy}{L\'{e}vy }
\newcommand{\upd}{\mathrm{d}}

\DeclareMathOperator*{\argmax}{arg\,max}
\DeclareMathOperator*{\argmin}{arg\,min}

\allowdisplaybreaks
\makeatother

\begin{document}

%\newpage{}

%\thispagestyle{empty}
%\tableofcontents
%\newpage{}

%\pagenumbering{arabic}

\title{Parametric Estimation of Tempered Stable Laws} 
\author[a]{Till Massing\thanks{Faculty of Economics, University of
Duisburg-Essen, Universit{\"{a}}tsstr.~12, 45117 Essen, Germany.\\E-Mail:
till.massing@uni-due.de}}

\maketitle

\begin{abstract}
Tempered stable distributions are frequently used in financial applications (e.g., for option pricing) in which the tails of stable distributions would be too heavy. Given the non-explicit form of the probability density function, estimation relies on numerical algorithms which typically are time-consuming. We compare several parametric estimation methods such as the maximum likelihood method and different generalized method of moment approaches. We study large sample properties and derive consistency, asymptotic normality, and asymptotic efficiency results for our estimators. Additionally, we conduct simulation studies to analyze finite sample properties measured by the empirical bias, precision, and asymptotic confidence interval coverage rates and compare computational costs. We cover relevant subclasses of tempered stable distributions such as the classical tempered stable distribution and the tempered stable subordinator. Moreover, we discuss the normal tempered stable distribution which arises by subordinating a Brownian motion with a tempered stable subordinator. Our financial applications to log returns of asset indices and to energy spot prices illustrate the benefits of tempered stable models.
\end{abstract}

\keywords{Tempered stable distributions; Parametric estimation; Maximum likelihood method; Generalized Method of Moments}\\
\classification{60E07;60G51;62F12;62P20}
\microtypesetup{activate=true}

\section{Introduction}\label{sec:intro}

We discuss parametric estimation methods for some well-known subclasses of tempered stable distributions. Estimation relies heavily on numerical methods as the probability density function is not given in closed form. This paper aims to compare available estimation methods both from an analytical as well as from a practical point of view. 

Tempered stable distributions are relevant from both a theoretical perspective and in the context of financial applications. Since tempered stable distributions are infinitely divisible, they can be used as the underlying marginal distribution for tempered stable \Levy processes. They arise by tempering the \Levy measure of stable distributions with a suitable tempering function. Tempered stable distributions were introduced in \cite{Koponen1995}, where the associated \Levy process was called smoothly truncated \Levy flight, which itself is a generalization of Tweedie distributions \cite[]{Tweedie1984}. Since then, tempered stable distributions have been generalized in several directions by mainly generalizing the class of tempering functions. \cite{Rosinski2007} and \cite{Rosinski2010} present a general framework for tempered stable distributions which contain the parametric subclasses to be considered in this paper. Further developments are surveyed in \cite{Grabchak2016book}.

The three subclasses we consider in this paper are the tempered stable subordinator (a one-sided distribution with finite variation), the classical tempered stable distribution (with the classical exponential tempering), and the normal tempered stable distribution (which is a normal-variance mixture with a tempered stable subordinator). The CGMY distribution \cite[]{CGMY2002} is a well-known special case of the classical tempered stable distribution which was introduced to model log-returns of stock prices. Tempered stable distributions have frequently been used for financial applications \cite[]{Kim2008,rachev2011financial,Fallahgoul2019}. Furthermore, \cite{Kim2008} and \cite{Kuchler2014} propose using tempered stable distributions for option pricing because the tails of these distributions are not too heavy for modeling financial returns (contrary to stable distributions). Besides financial applications, tempered stable distributions have been used for many other domains, for example for modeling cell generation \cite[]{Palmer2008}, internet traffic \cite[]{Terdik2008}, and solar-wind velocity \cite[]{Bruno2004}.

 Estimation methods for generalized tempered stable distributions are still an active area of research. We compare various well-established estimation methods in the literature. The first is the traditional maximum likelihood (ML) method, which works by numerical optimization and Fourier inversion, see \cite{Kim2008,rachev2011financial}. \cite{Grabchak2016} proves strong consistency of the maximum likelihood estimator (MLE). \cite{Kuchler2013} propose moment methods which are easier and faster than the MLE. We also use the generalized method of moments estimator by \cite{Feuerverger1981} that is based on empirical characteristic functions and the generalized method of moments on a continuum of moment conditions by \cite{Carrasco2017}. The latter method already turned out to be useful in estimating stable distributions, see \cite{Garcia2011stable}. Further available methods include the method of simulated quantiles \cite[]{Dominicy2013,Fallahgoul2019quantile}, and non- or semiparametric methods \cite[]{Belomestny2015,Figueroa2022}.

This paper contributes to the literature in three ways. First, we derive asymptotic theory for the ML method and the generalized method of moments for the three classes of tempered stable distributions. More precisely, we prove asymptotic efficiency and asymptotic normality of the estimators by verifying a set of sufficient conditions. Second, we compare finite sample properties of the estimators in a Monte Carlo study. Third, we illustrate that tempered stable distributions are more suitable in financial applications than stable distributions because the tails of the latter are too heavy. For this, we study log-returns of three financial time series namely the S\&P 500, the German DAX, and the German EEX electricity spot prices.

The remainder of this paper is organized as follows. Section \ref{sec:distributions} presents formal definitions and properties of tempered stable distributions and some of their important subclasses. Section \ref{sec:estimators} discusses the estimation strategies and states their general asymptotic results. Section \ref{sec:results} contains our theoretical results. All proofs are relegated to the appendix. In Section \ref{sec:MCstudy} we conduct a simulation study to analyze finite sample properties. We discuss financial applications in Section \ref{sec:application}. Section \ref{sec:conclusion} concludes.

\section{Tempered stable distributions}\label{sec:distributions}

To establish notation, we describe some general properties of tempered stable distributions and the considered special cases in this paper. \citet[Section 8]{Sato1999levy} compares different versions of the L\'evy-Khintchine representation to uniquely describe distributions of infinitely divisible random variable $X$. In this paper, we make use of two of these versions. In our setting, it depends on the index of stability parameter $\alpha\in(0,2)$ which version is being used. The first version describes infinitely distributions by the characteristic triple $(\mu,\sigma^2,\Pi)_1$, such that
\begin{equation}\label{eq:khintchine}
\mathbb{E}\left[\mathrm{e}^{\mathrm{i} t X}\right]=\exp\left(\mathrm{i} t\mu-\frac{1}{2} \sigma^2 t^2+\int_{\mathbb{R}}\left(\mathrm{e}^{\mathrm{i} tr}-1-\mathrm{i} tr\right)\Pi(\mathrm{d}r)\right),
\end{equation}
where $\mu\in\mathbb{R}$, $\sigma\ge0$, and $\Pi$ is a measure on $\mathbb{R}$ called \emph{L\'evy measure} satisfying $\Pi(\{0\})=0$ and
\[\int_{\mathbb{R}}(|r|^2\wedge |r|)\Pi(\mathrm{d}r)<\infty,\]
which ensures that $\Pi$ is $\sigma$-finite. We use this characterization in the case of $\alpha\in(1,2)$.

The second decomposition is characterized by the triple $(\mu_0,\sigma^2,\Pi)_0$, such that
\begin{equation}\label{eq:Laplaceexponent}
\mathbb{E}\left[\mathrm{e}^{\mathrm{i} t X}\right]=\exp\left(\mathrm{i} t\mu_0-\frac{1}{2} \sigma^2 t^2+\int_0^{\infty}\left(\mathrm{e}^{\mathrm{i} tr}-1\right)\Pi(\mathrm{d}r)\right).
\end{equation}
$\mu_0$ is called drift and $\Pi$ is the L\'evy measure, as long as $\mu_0\ge0$ and $\int_{0}^{\infty}(r\wedge 1)\Pi(\mathrm{d}r)<\infty$. We use this characterization for $\alpha\in(0,1)$. We omit to give a parametrization which covers the case $\alpha=1$ and refer to \citet[Section 8]{Sato1999levy}

A \emph{subordinator} is a one-dimensional, (a.s.)~non-decreasing L\'{e}vy process. For subordinators we use the decomposition $(\mu_0,\sigma^2,\Pi)_0$, where $\Pi((-\infty,0])=0$. Although formally there is a difference between stochastic processes and distributions, we will for simplicity refer to the distribution of this subordinator process at time 1 as a subordinator as well. Throughout this paper, when we talk about L\'evy processes we are mainly interested in the characterizing distribution at time 1.

Important special cases are so-called stable (or $\alpha$-stable) \Levy processes (see \cite{Sato1999levy} or \cite{Nolan2020}). One-dimensional stable processes are characterized by the \Levy measure
%\begin{equation}
%\label{eq:LevyStable}
%M(\upd r, \upd u) = r^{-1-\alpha}\upd r \sigma(\upd u),
%\end{equation}
%where $\alpha\in(0,2)$ and $\sigma$ is a finite, non-zero measure on the $(d-1)$-dimensional unit sphere $S^{d-1}$. In one dimension the \Levy measure takes the form
\begin{equation}
\label{eq:LevyStable1}
M(\upd r) = \left(\frac{\delta_+}{r^{1+\alpha}}\mathds{1}_{(0,\infty)}(r)+\frac{\delta_-}{|r|^{1+\alpha}}\mathds{1}_{(-\infty,0)}(r)\right)\upd r,
\end{equation}
with $\alpha\in(0,2)$, where $\alpha$ is called the index of stability, and $\delta_+,\delta_-\ge0$ s.t.~$(\delta_+,\delta_-)\neq(0,0)$. We call $\rho=\frac{\delta_+-\delta_-}{\delta_++\delta_-}\in[-1,1]$ the skewness parameter. 

Tempered stable distributions arise by tempering the \Levy measure of a stable distribution by a tempering function. The \Levy measure is
\begin{equation}
\label{eq:LevyTS}
Q(\upd r) =\left(\frac{\delta_+q(r,+1)}{r^{1+\alpha}}\mathds{1}_{(0,\infty)}(r)+\frac{\delta_-q(|r|,-1)}{|r|^{1+\alpha}}\mathds{1}_{(-\infty,0)}(r)\right)\upd r
\end{equation}
where $q:(0,\infty)\times \{\pm1\}\rightarrow (0,\infty)$ is a Borel function. \cite{Rosinski2007} considered the case where $q(\cdot,u)$ is completely monotone with $\lim_{r\to\infty}q(r,u)=0$ for each $u\in \{\pm1\}$, i.e., $(-1)^n\frac{\partial^n}{\partial r^n}q(r,u)>0$ for all $r>0,u\in \{\pm1\}, n\in\mathbb{N}_0$. It is called a proper tempered stable distribution if, in addition, $\lim_{r\downarrow0}q(r,u)=1$ for $u\in \{\pm1\}$. Proper tempered stable distributions follow the initial motivation by modifying the tails of stable distributions to make them lighter. \cite{Rosinski2010} generalize tempered stable distributions by relaxing the complete monotonicity assumption and allowing $q$ to only converge in a certain sense to some non-negative function $g$ (see \cite{Rosinski2010} for details). A number of parametric forms for $q$ have been proposed in the literature, see \cite{rachev2011financial} for some examples. In this paper, we mainly focus on the exponential (or classical) tempering function where we apply an exponentially decreasing function. See the next subsections for details. For both stable and tempered stable distributions, the Gaussian term $\sigma^2$ is zero.

\subsection{Tempered stable subordinator}\label{subsec:TSS}

The first special case we discuss is the tempered stable subordinator (TSS). It is constructed from the stable subordinator which is a non-negative, increasing \Levy process with $\alpha$-stable marginals. In this case, the stability parameter $\alpha$ needs to be in $(0,1)$. Considering the characterization \eqref{eq:Laplaceexponent} with parametrization $(\mu_0,\sigma^2,\Pi)_0$, the \Levy measure of a stable subordinator is 
\begin{equation}\label{eq:LevySS}
\frac{\delta}{r^{1+\alpha}}\mathds{1}_{(0,\infty)}(r)\upd r,
\end{equation}
where $\delta>0$ is a scale parameter, and the drift $\mu_0$ and $\sigma^2$ are zero. For exponential tempering, the \Levy measure of the TSS distribution is given by
\begin{equation}\label{eq:LevyTSS}
Q_{TSS}(\upd r)=\frac{\mathrm{e}^{-\lambda r}\delta}{r^{1+\alpha}}\mathds{1}_{(0,\infty)}(r)\upd r,
\end{equation}
where $\lambda>0$ is the tempering parameter.

Let $Y\sim TSS(\alpha,\delta,\lambda)$ be a TSS distributed random variable defined on $\mathbb{R}_+$. Using \eqref{eq:LevyTSS} it is possible to derive the characteristic function of the TSS distribution. See \citet[][Lemma 2.5]{Kuchler2013}, for a proof of
\begin{equation}
\label{eq:charTSS}
\varphi_{TSS}(t;\theta):=\mathbb{E}_{\theta}\left[\mathrm{e}^{\mathrm{i}tY}\right]= \exp\left(\delta\Gamma(-\alpha)\left((\lambda-\mathrm{i}t)^{\alpha}-\lambda^{\alpha}\right)\right),
\end{equation}
with parameter vector $\theta=(\alpha,\delta,\lambda)$, where the power stems from the main branch of the complex logarithm. $\mathbb{E}_{\theta}$ is the expectation operator w.r.t.~the data generating process indexed by $\theta$.

The probability density function of tempered stable distributions is generally not available in closed form. For the density of the TSS distribution we can make use of the identity
\begin{equation}
\label{eq:dTSS}
f_{TSS}(y;\theta)=\mathrm{e}^{-\lambda y-\lambda^{\alpha}\delta\Gamma(-\alpha)}f_{S(\alpha,\delta)}(y),
\end{equation}
see \citet[][eq.~(2.6)]{Kawai2011}. $S(\alpha,\delta)$ denotes the distribution of the $\alpha$-stable subordinator with scale parameter $\delta$ and $f_{S(\alpha,\delta)}(y)$ denotes its density. Both $f_{TSS}(y;\theta)$ and $f_{S(\alpha,\delta)}(y)$ are only defined on $\mathbb{R}_+$. $f_{S(\alpha,\delta)}(y)$ is (except for a few special cases) not available in closed form. However, many software packages (like the \texttt{stabledist} or \texttt{Tweedie} packages in R) have fast computation routines based on series or integral representations. Combining such series representation with \eqref{eq:dTSS}, we obtain a series representation for the TSS distribution
\begin{equation}
\label{eq:dTSSseries}
f_{TSS}(y;\theta)=\mathrm{e}^{-\lambda y-\lambda^{\alpha}\delta\Gamma(-\alpha)}\frac{-1}{\pi}\sum_{k=1}^{\infty}\frac{(-1)^k}{k!}\Gamma(1+\alpha k)\Gamma(1-\alpha)^k\left(\frac{\delta}{\alpha}\right)^ky^{-(1+\alpha k)}\sin(\alpha\pi k),
\end{equation}
see \cite{Bergstrom1952,Nolan2020}.

The estimation method in Section \ref{subsec:GCM} makes use of matching theoretical with empirical cumulants. Therefore, we state the cumulant generating function of the TSS distribution 
\begin{equation}\label{eq:cgfTSS}
\psi_{TTS}(t;\theta)= \delta\Gamma(-\alpha)\left((\lambda-t)^{\alpha}-\lambda^{\alpha}\right),
\end{equation}
for $t\le\lambda$, derived in \cite{Kuchler2013}. Thus, the $m$-th order cumulants $\kappa_m=\left.\frac{\upd^m}{\upd t^m}\psi(t)\right|_{t=0}$ are given by
\begin{equation}\label{eq:cumsTSS}
\kappa_m=\Gamma(m-\alpha)\frac{\delta}{\lambda^{m-\alpha}},\ \ \ m\in\mathbb{N}.
\end{equation}

Simulation, which we need in the Monte Carlo study, of TSS distributed random variates is straightforward by an acceptance-rejection algorithm, i.e., we first generate $U\sim\mathcal{U}(0,1)$ and $V\sim S(\alpha,\delta)$. If $U\le \mathrm{e}^{-\lambda V}$ we set $Y:=V$, otherwise we repeat the first step, see \cite{Kawai2011}. See \cite{Hofert2011} for the more efficient double rejection method. For the generation of stable random numbers see, e.g., \cite{Nolan2020}.

\subsection{Centered and totally positively skewed tempered stable distribution}\label{subsec:PTS}

Closely related to the TSS distribution is the centered and totally positively skewed tempered stable distribution $TS^{\prime}(\alpha,\delta,\lambda)$ which is defined by its characteristic function
\begin{equation}
\label{eq:charPTS}
\exp\left(\int_{\mathbb{R}_+}\left(\mathrm{e}^{\mathrm{i} tr}-1-\mathrm{i} tr\right)\frac{\mathrm{e}^{-\lambda r}\delta}{r^{1+\alpha}}\upd r\right)=\exp\left(\delta\Gamma(-\alpha)\left((\lambda-\mathrm{i}t)^{\alpha}-\lambda^{\alpha}+\mathrm{i}t\alpha\lambda^{\alpha-1}\right)\right),
\end{equation}
for $\alpha\in(0,2)$, where the power stems from the main branch of the complex logarithm. Note that here we use the characterization \eqref{eq:khintchine} with parametrization $(\mu,\sigma^2,\Pi)_1$ instead of \eqref{eq:Laplaceexponent} although the \Levy measure $\frac{\mathrm{e}^{-\lambda r}\delta}{r^{1+\alpha}}\upd r$ is the same. For $\alpha\in(0,1)$, we have the relation that if $Y\sim TSS(\alpha,\delta,\lambda)$, then $Y-\delta\Gamma(1-\alpha)\lambda^{\alpha-1}\sim TS^{\prime}(\alpha,\delta,\lambda)$. In particular, 
\begin{equation}
\label{eq:denscomposition}
f_{TS^{\prime}}(y;\alpha,\delta,\lambda)=f_{TSS}(y-\Gamma(1-\alpha)\delta\lambda^{\alpha-1};\alpha,\delta,\lambda).
\end{equation}
For $\alpha\in(0,1)\cup(1,2)$, we additionally have
\begin{equation}
\label{eq:denscomposition2}
f_{TS^{\prime}}(y;\alpha,\delta,\lambda)=\mathrm{e}^{-\lambda y-\lambda^{\alpha}\delta(\alpha+1)\Gamma(-\alpha)}f_{S(\alpha,\delta)}(y-\Gamma(1-\alpha)\delta\lambda^{\alpha-1}),
\end{equation}
where $S(\alpha,\delta)$ is the totally positively skewed stable distribution with \Levy measure $\delta r^{-\alpha-1}\mathds{1}_{(0,\infty)}(r)\upd r$ and characteristic function
\begin{equation}
\label{eq:charPS}
\exp\left(-\delta\frac{\Gamma(1-\alpha)}{\alpha}\cos(\pi\alpha/2)|t|^{\alpha}\left(1-\mathrm{i}\tan(\pi\alpha/2)\mathrm{sgn}(t)\right)\right),
\end{equation}
which is the same as the stable subordinator if $\alpha\in(0,1)$.
Note that for $\alpha\in(0,1)$ the TS' distribution is defined on $(\Gamma(1-\alpha)\delta\lambda^{\alpha-1},\infty)$ instead of $\mathbb{R}_+$ as for the TSS distribution. This in fact makes classical asymptotic theory for the MLE infeasible and we only use the distribution as a tool for proving results about CTS distributions to be defined in Subsection \ref{subsec:CTS}. For $\alpha\ge1$, the TS' distribution is defined on $\mathbb{R}$.

Simulation of totally positively skewed tempered stable random variables is more involved than for the subordinator as the simple acceptance-rejection does not work for $\alpha\in(1,2)$. \cite{Kawai2011} present several remedies, e.g., a truncated series representation by \cite{Rosinski2001}. We opt for the simulation approach of \cite{Baeumer2010}, i.e., using an approximate acceptance-rejection algorithm which works as follows: first fix a number $c>0$. Second, simulate $U\sim\mathcal{U}(0,1)$ and $V\sim S(\alpha,\delta)$. If $U\le \mathrm{e}^{-\lambda (V+c)}$ we set $Y:=V -\Gamma(1-\alpha) \delta\lambda^{\alpha-1}$, otherwise we return to the second step. The algorithm is not exact, i.e., $Y\nsim TS^{\prime}(\alpha,\delta,\lambda)$. The number $c$ controls the degree of approximation and also the acceptance rate. For too small $c$, the approximation might not be sufficient. For large $c$, the approximation improves; yet, the acceptance probability decreases and therefore the runtime elongates.

\subsection{Classical tempered stable distribution}\label{subsec:CTS}

Next, we discuss one-dimensional classical tempered stable (CTS) distributions. They are defined by their \Levy measure 
\begin{equation}\label{eq:LevyCTS}
Q_{CTS}(\upd r)=\left(\frac{\mathrm{e}^{-\lambda_+r}\delta_+}{r^{1+\alpha}}\mathds{1}_{(0,\infty)}(r)+\frac{\mathrm{e}^{-\lambda_-|r|}\delta_-}{|r|^{1+\alpha}}\mathds{1}_{(-\infty,0)}(r)\right)\upd r
\end{equation}
in representation \eqref{eq:khintchine} with parametrization $(\mu,\sigma^2,\Pi)_1$.
$\alpha\in(0,2)$ is the stability parameter, $\delta_+,\delta_->0$ are scaling parameters, $\lambda_+,\lambda_->0$ are tempering parameters and $\mu$ is a location parameter. The indices $+$ and $-$ refer to the positive and negative parts of the distribution (centered around $\mu$). We collect all parameters in the vector $\theta=(\alpha,\delta_+,\delta_-,\lambda_+,\lambda_-,\mu)$. Note that for the CTS distribution the parameter vector $\theta$ is different than for the TSS distribution.

Let $X\sim CTS(\alpha,\delta_+,\delta_-,\lambda_+,\lambda_-,\mu)$, which is a distribution on $\mathbb{R}$. The characteristic function is given by
\begin{align}\label{eq:charCTS}
\varphi_{CTS}(t;\theta):=\mathbb{E}_{\theta}\left[\mathrm{e}^{\mathrm{i}tX}\right]&=\exp\left(\mathrm{i}t\mu+\delta_+\Gamma(-\alpha)\left((\lambda_+-\mathrm{i}t)^{\alpha}-\lambda_+^{\alpha}+\mathrm{i}t\alpha\lambda_+^{\alpha-1}\right)\right.\\
&\ \left. +\delta_-\Gamma(-\alpha)\left((\lambda_-+\mathrm{i}t)^{\alpha}-\lambda_-^{\alpha}-\mathrm{i}t\alpha\lambda_-^{\alpha-1}\right)\right),
\end{align}
for all $\theta\in(0,2)\times(0,\infty)^4\times\mathbb{R}$ such that $\alpha\neq1$. When $\alpha=1$, for $\theta_1=(1,\delta_+,\delta_-,\lambda_+,\lambda_-,\mu)$ the characteristic function of the CTS distribution has the form
\begin{align}\label{eq:charCTS1}
\varphi_{CTS}(t;\theta)&=\exp\left(\mathrm{i}t\mu+\delta_+\left((\lambda_+-\mathrm{i}t)\log(1-\mathrm{i}t/\lambda_+)+\mathrm{i}t\right)\right.\\
&\ \left. +\delta_-\left((\lambda_-+\mathrm{i}t)\log(1+\mathrm{i}t/\lambda_-)-\mathrm{i}t\right)\right).
\end{align}
Note that the characteristic function and the density function are continuous in $\alpha\in(0,2)$.
%Note that the characteristic function (and hence the density function) is discontinuous at $\alpha=1$, which affects the compact domain for which asymptotic analysis is possible \cite[see][]{DuMouchel1973}.

As for the TSS distribution, the density function of the CTS distribution does not exist in closed form. Crucially, even a simple relationship with a stable density as for the TSS distribution in \eqref{eq:dTSS} is not available. For numerical evaluations it is therefore necessary to rely on algorithms like the fast Fourier transform \cite[FFT, see][]{Brigham1988} applied to the characteristic function \eqref{eq:charCTS}.

As for the TSS distribution, we specify the cumulant generating function
\begin{align}\label{eq:cgfCTS}
\psi_{CTS}(t;\theta)&= t\mu+\delta_+\Gamma(-\alpha)\left((\lambda_+-t)^{\alpha}-\lambda_+^{\alpha}+t\alpha\lambda_+^{\alpha-1}\right)\\
&\  +\delta_-\Gamma(-\alpha)\left((\lambda_-+t)^{\alpha}-\lambda_-^{\alpha}-t\alpha\lambda_-^{\alpha-1}\right),
\end{align}
for $t\in[-\lambda_-,\lambda_+]$. We use theoretical cumulants for cumulant matching below. The $m$-th order cumulants can be derived from \eqref{eq:cgfCTS} and take the form
\begin{equation}\label{eq:cumsCTS}
\kappa_m=\Gamma(m-\alpha)\frac{\delta_+}{\lambda_+^{m-\alpha}}+(-1)^m\Gamma(m-\alpha)\frac{\delta_-}{\lambda_-^{m-\alpha}},
\end{equation}
for $m\ge 2$ and $\kappa_1=\mu$.

%To explain a feasible simulation algorithm we first discuss the relationship between the CTS distribution and the TSS distribution of the previous subsection by introducing a related distribution. The relation to the CTS distribution will become apparent in the following.

CTS distributed random variables can be constructed from totally positively skewed tempered stable random variables in the following way. Let $Y_+\sim TS^{\prime}(\alpha,\delta_+,\lambda_+)$ and $Y_-\sim TS^{\prime}(\alpha,\delta_-,\lambda_-)$ be independent and $\mu\in\mathbb{R}$. Then
\begin{equation}
\label{eq:convolutionCTS}
X:= Y_+ - Y_- +\mu \sim CTS(\alpha,\delta_+,\delta_-,\lambda_+,\lambda_-,\mu).
\end{equation}

\subsection{Normal tempered stable distribution}\label{subsec:NTS}
Another model that is often used in financial applications is the normal tempered stable (NTS) distribution. It is constructed as a classical normal variance mixture, see \cite{barndorff2001normal}. For this, let $Y\sim TSS(\alpha,\delta,\lambda)$, with $(\alpha,\delta,\lambda)\in(0,1)\times(0,\infty)^2$. Let $B\sim N(0,1)$ be independent of $Y$ and $\rho,\mu\in\mathbb{R}$.
Set
\begin{equation}
\label{eq:NTSsubordination}
Z=\sqrt{Y}B+\beta Y+\mu.
\end{equation}
Then, $Z$ is $NTS(\theta)$ distributed, where for this case $\theta=(\alpha,\beta,\delta,\lambda,\mu)$.
We can also obtain the NTS distribution by tempering a stable distribution. The corresponding tempering function can be found in \citet[][Table 3.4]{rachev2011financial}. Note that the tempering function is not completely monotone but it is in the class of generalized tempered stable distributions of \cite{Rosinski2010}.

For our parametrization, the characteristic function now takes the form
\begin{equation}
\label{eq:charNTS}
\varphi_{NTS}(t;\theta)=\mathbb{E}_{\theta}\left[\mathrm{e}^{\mathrm{i}tZ}\right]= \exp\left(\mathrm{i}t\mu+\delta\Gamma(-\alpha)\left((\lambda-\mathrm{i}t\beta+t^2/2)^{\alpha}-\lambda^{\alpha}\right)\right),
\end{equation}
where the power stems from the main branch of the complex logarithm. As for the CTS distribution, the density function is not available in closed form and numerical computation relies on numerical methods such as FFT. 

In this case the cumulants do not have an easy pattern as for the other examples and so we omit them. We also do not propose a cumulant matching estimation method here.

Simulation of NTS distributed random variables is easy given independent TSS and standard normal random variables by invoking \eqref{eq:NTSsubordination}.

\section{Estimation Methods}\label{sec:estimators}
This section discusses some parametric estimation strategies available in the literature. We apply these to the tempered stable distributions considered above and derive asymptotic efficiency and normality in the next section. In this section, we briefly present the methods and some known general asymptotic results. Throughout this section let $X$ be a random variable following one of the tempered stable distributions of Section \ref{sec:distributions} and let $f(x;\theta)$ denote its density function, depending on the parameter vector $\theta$. Also, denote by
$\varphi_{\theta}(t)$
its characteristic function. Let $\theta_0$ be the unknown true parameter vector. In this paper, we only consider the case of an i.i.d.~sample $X_1,\ldots,X_n$ with density function $f(x;\theta_0)$.

\subsection{Maximum likelihood estimation}\label{subsec:MLE}

Maximum likelihood estimation is standard in the literature and frequently used, for example in \cite{Kim2008}. We numerically maximize the log-likelihood function
\begin{equation}
\label{eq:loglikelihood}
\ell_n(\theta)=\sum_{j=1}^n\log f(X_j;\theta)
\end{equation}
with respect to $\theta$ to find the MLE 
\begin{equation}
\label{eq:MLE}
\hat{\theta}_{n,ML}=\argmax_{\theta\in\Theta} \ell(\theta).
\end{equation}

As described in the preceding section, the density functions of our distributions are not available in closed form but either via a series representation (TSS) or via the Fourier inversion (CTS and NTS)
\begin{equation}
\label{eq:FourierInv}
f(x;\theta)=\frac{1}{2\pi}\int_{\mathbb{R}}\mathrm{e}^{-\mathrm{i}tx}\varphi_{\theta}(t)\upd t
\end{equation}
based on the characteristic function $\varphi_{\theta}(t)$, which is feasible because \eqref{eq:charCTS} and \eqref{eq:charNTS} are integrable. % \cite[see, e.g.,][Theorem 1.1]{Nolan2020}.
In practice, we use the FFT algorithm to approximate \eqref{eq:FourierInv}.

Among others, \citet[][Theorem 3.3]{Newey1994} (which we here follow) proved the limiting behavior of the maximum likelihood estimator as given in Proposition \ref{prop:MLasymnorm} under the following assumption.
\begin{ass}\label{ass:ML}
\item[(i)] $\hat{\theta}_{n,ML}$ is consistent for $\theta_0$.
\item[(ii)] $\theta_0$ is an interior point of $\Theta$ which is compact.
\item[(iii)] $f(x;\theta)$ is twice continuously differentiable in $\theta$ in a neighborhood $\mathcal{N}$ around $\theta_0$ and the support of $f$ is equal to the entire domain $\mathcal{D}$ and does not depend on $\theta$.
\item[(iv)] $\int_{\mathcal{D}}\sup_{\theta\in\mathcal{N}}\Big|\Big|\frac{\partial f(x;\theta)}{\partial \theta}\Big|\Big|\upd x<\infty$, $\int_{\mathcal{D}}\sup_{\theta\in\mathcal{N}}\Big|\Big|\frac{\partial^2 f(x;\theta)}{\partial \theta\partial\theta^{\prime}}\Big|\Big|\upd x<\infty$, with $\mathcal{N}$ as in (iii).
\item[(v)] $I_{\theta_0}=\left.\mathbb{E}_{\theta_0}\left[\left(\frac{\partial \log f(X;\theta)}{\partial \theta}\right)\left(\frac{\partial\log f(X;\theta)}{\partial \theta}\right)^{\prime}\right]\right|_{\theta=\theta_0}$ is positive definite.
\item[(vi)] $\mathbb{E}_{\theta_0}\left[\sup_{\theta\in\mathcal{N}}\Big|\Big|\frac{\partial^2 \log f(X;\theta)}{\partial \theta\partial\theta^{\prime}}\Big|\Big|\right]<\infty$, with $\mathcal{N}$ as in (iii).
\end{ass}

The norm $||\cdot||$ for vectors is the usual Euclidean norm and for matrices the Frobenius norm (which can be seen as an Euclidean norm for matrices). Some other references for asymptotic normality results under different sets of assumptions are \cite{Cramer1946} or \cite{Le1956}. 

The stated form of (iii) is slightly different than (ii) in \citet[][Theorem 3.3]{Newey1994} who require the density to be positive for all $x\in\mathbb{R}$. We use the relaxed assumption here because by nature subordinators have no positive density for negative $x$. This relaxation, however, is still within the scope of \cite{Newey1994} because their assumption implies the more general requirement of \citet[][Theorem 3.1]{Newey1994} that the objective function for maximization $\frac{1}{n}\sum_{j=1}^n\log f(x_j,n)$ is twice continuously differentiable in a neighborhood $\mathcal{N}$ of $\theta_0$. This also holds for (iii) above. The essential point is that the density is that the support does not depend on $\theta$ which rules out the TS' distribution for $\alpha<1$. 

We remark that we only need to assume the existence of one neighborhood $\mathcal{N}$. It is not necessary that (iii), (iv) and (vi) hold for any neighborhood around $\theta_0$.

\begin{prop}\label{prop:MLasymnorm}
Under Assumption \ref{ass:ML}, $\hat{\theta}_{n,ML}$
is consistent and
\begin{equation}
\label{eq:MLasymnorm}
n^{1/2}(\hat{\theta}_{n,ML}-\theta_0)\stackrel{\mathcal{L}}{\rightarrow}N(0,I_{\theta_0}^{-1}),
\end{equation}
as $n\to\infty$, where $I_{\theta_0}^{-1}$ denotes the inverse of the Fisher information matrix (which exists due to Assumption \ref{ass:ML}(v)).
\end{prop}

\subsection{Generalized Method of Moments}\label{subsec:GMM}

The generalized method of moments (GMM) by \cite{Hansen1982} is suitable for estimating tempered stable laws. One approach to define moment conditions is to use the theoretical characteristic function $\varphi_{\theta}(t)$ of $X$.
The sample analogue for the realizations $\{X_j\}_{j=1,\ldots,n}$ is
\begin{equation}
\label{eq:empchar}
\hat{\varphi}_n(t)=\frac{1}{n}\sum_{j=1}^n \mathrm{e}^{\mathrm{i}tX_j}.
\end{equation}
We form moment conditions
\begin{equation}
\label{eq:theormom}
\mathbb{E}_{\theta_0}\left[h(t,X_j;\theta)\right]=0
\end{equation}
for all $t\in \mathbb{R}$, where
\begin{equation}
\label{eq:momcond}
h(t,X_j;\theta)=\mathrm{e}^{\mathrm{i}tX_j}-\varphi_{\theta}(t).
\end{equation}
The sample analogue is denoted by
\begin{equation}
\label{eq:empmom}
\hat{h}_n(t;\theta)=\frac{1}{n}\sum_{j=1}^n h(t,X_j;\theta)=\hat{\varphi}_n(t)-\varphi_{\theta}(t).
\end{equation}

We now review some approaches on how to choose a set of $t$'s to obtain appropriate moment conditions. One way is to choose a finite grid $\{t_1,\ldots,t_R\}\subset\mathbb{R}$, where $R$ denotes the grid size. Given the grid, we define a vector-valued function $g(X_j;\theta)=\left(h(t_1,X_j;\theta),\ldots,h(t_R,X_j;\theta)\right)^{\mathrm{T}}$. We then minimize the objective function 
\begin{equation}\label{eq:GMM}
\left(\frac{1}{n}\sum_{j=1}^ng(X_j;\theta)\right)^{\prime}\hat{W}\left(\frac{1}{n}\sum_{j=1}^ng(X_j;\theta)\right)
\end{equation}
in $\theta$, where we choose $\hat{W}=\hat{\Omega}^{-1}$ so that the asymptotic variance is optimal.
\cite{Feuerverger1981} show that the asymptotic variance of the estimator can be made arbitrarily close to the Cram\'er-Rao bound by selecting the grid sufficiently fine. However, as argued by \cite{Carrasco2017}, the grid size $R$ must not be larger than the sample size. Otherwise, the problem becomes ill-posed since the asymptotic variance matrix of the moment conditions becomes singular. \cite{Carrasco2017} generalize the empirical characteristic function GMM approach by introducing an estimator based on a continuum of moment conditions (CGMM). They derive that the asymptotic variance attains the Cram\'er-Rao bound. They solve the singularity issue of the asymptotic variance matrix by applying a suitable regularization. We discuss this approach in more detail below. For the case of the GMM estimator based on a discrete set of moment conditions, we follow \cite{Kharrat2016} in the numerical computations and also use a regularization to make the scheme numerically stable.

Next, we describe the CGMM estimation method of \cite{Carrasco2017}, which is based on \cite{Carrasco2000} and \cite{Carrasco2007}. We start by introducing some notation. Let $\pi$ be a probability density on $\mathbb{R}$ and $L^2(\pi)$ be the Hilbert space of complex-valued functions such that
\begin{equation}
\label{eq:HilbertL}
L^2(\pi)=\left\{f: \mathbb{R}\rightarrow\mathbb{C} \ : \ \int|f(t)|^2\pi(t)\upd t<\infty\right\}.
\end{equation}
The inner product on $L^2(\pi)$ is defined as
\begin{equation}
\label{eq:innerproduct}
\left\langle f,g\right\rangle_{L^2(\pi)}=\int f(t) \overline{g(t)}\pi(t)\upd t
\end{equation}
and the norm on $L^2(\pi)$ as
\begin{equation}
\label{eq:norm}
||g||_{L^2(\pi)}^2=\int |g(t)|^2\pi(t)\upd t.
\end{equation}
Let $K$ be the asymptotic variance-covariance operator associated with the moment functions $h(t,X;\theta)$. $K$ is an integral operator that satisfies
\begin{align}
\label{eq:operatorK}
K:\ L^2(\pi)&\rightarrow L^2(\pi)\\
\ f&\mapsto g,\ \mathrm{where}\ g(t)=\int k(s,t)f(s)\pi(s)\upd s,
\end{align}
where $k(s,t)$ is a kernel given by
\begin{equation}
\label{eq:kernelk}
k(s,t)=\mathbb{E}_{\theta_0}\left[h(s,X;\theta_0)\overline{h(t,X;\theta_0)}\right],
\end{equation}
with $h$ given in \eqref{eq:momcond}. \cite{Carrasco2017} noted that the inverse of $K$ exists only on a dense subset of $L^2(\pi)$. Thus, we use a regularized estimation of the inverse below. The efficient CGMM estimator is given by
\begin{equation}
\label{eq:CGMM}
\hat{\theta}=\argmin_{\theta\in\Theta} \left\langle K^{-1}\hat{h}_n(\cdot;\theta),\hat{h}_n(\cdot;\theta)\right\rangle_{L^2(\pi)}.
\end{equation}

The above CGMM is non-feasible because we need an estimate $\hat K_n$ for $K$. To get a feasible estimator we first need to estimate $k(s,t)$ in \eqref{eq:kernelk} with
\begin{equation}
\label{eq:estkernelk}
\hat k_n(s,t)=\frac{1}{n}\sum_{j=1}^n\left(\mathrm{e}^{\mathrm{i}s^{\prime}X_j}-\hat{\varphi}_n(s)\right)\left(\overline{\mathrm{e}^{\mathrm{i}t^{\prime}X_j}-\hat{\varphi}_n(t)}\right).
\end{equation}
Second, an empirical operator $\hat K_n$ with kernel function $\hat k_n(s,t)$ is defined by
\begin{align}
\label{eq:empoperatorK}
(\hat K_nf)(t)=\int \hat k(s,t)f(s)\pi(s)\upd s.
\end{align}
However, this choice is non-invertible. Therefore, \cite{Carrasco2017} estimate $K^{-1}$ by a Tikhonov regularization with
\begin{equation}
\label{eq:operatorKest}
\hat K_{n,\gamma_n}^{-1}=\left(\hat K_n^2+\gamma_n I\right)^{-1}\hat K_n.
\end{equation}
$\gamma_n$ is (depending on the sample size) a sequence of regularization parameters which allow $\hat K_{n,\gamma_n}^{-1}f$ to exist for all $f\in L^2(\pi)$ and to dampen the sensitivity of $\hat K_{n,\gamma_n}^{-1}f$ to variation in the input $f$.
Then, the feasible CGMM estimator is given by
\begin{equation}
\label{eq:fCGMM}
\hat{\theta}_{n,CGMM}(\gamma_n)=\argmin_{\theta\in\Theta} \left\langle \hat K_{n,\gamma_n}^{-1}\hat{h}_n(\cdot;\theta),\hat{h}_n(\cdot;\theta)\right\rangle_{L^2(\pi)}.
\end{equation}

\cite{Carrasco2017} show that the CGMM estimator is consistent, asymptotically efficient and asymptotically normal (for stationary Markov processes) given a set of assumptions. An earlier version \cite{Carrasco2002} proves the statement for i.i.d.~data with a simpler set of assumptions. We use their assumptions and prove that the tempered stable distributions fulfill them. More precisely, the assumptions are the following.
\begin{ass}\label{ass:CGMM}
\item[(i)] The observed data $\{x_1,\ldots,x_n\}$ are i.i.d.~realizations of $X$ which has values in $\mathbb{R}$ and has p.d.f.~$f(x;\theta)$ with $\theta\in\Theta\subset\mathbb{R}^q$ and $\Theta$ is compact.
\item[(ii)] $\pi$ is the p.d.f.~of a distribution that is absolutely continuous with respect to the Lebesgue measure and strictly positive for all $x\in\mathbb{R}$.% $L^2(\pi)$ is the Hilbert space of complex-valued functions that are square integrable with respect to $\pi$. 
\item[(iii)] The equation
\begin{equation}
\label{eq:momcondass}
\mathbb{E}_{\theta_0}\left[\mathrm{e}^{\mathrm{i}tX}\right]-\varphi_{\theta}(t)=0\ \mathrm{for}\ \mathrm{all}\ t\in\mathbb{R},\ \pi\mathrm{-a.s.}
\end{equation}
has a unique solution $\theta_0$ which is an interior point of $\Theta$. Since characteristic functions uniquely determine distributions \eqref{eq:momcondass} is equivalent to identifiability.
\item[(iv)] $f(x;\theta)$ is continuously differentiable with respect to $\theta$ on $\Theta$.
\item[(v)] $\int_{\mathcal{D}}\sup_{\theta\in\Theta}\Big|\Big|\frac{\partial f(x;\theta)}{\partial \theta}\Big|\Big|\upd x<\infty$.
\item[(vi)] $I_{\theta_0}=\left.\mathbb{E}_{\theta_0}\left[\left(\frac{\partial \log f(X;\theta)}{\partial \theta}\right)\left(\frac{\partial\log f(X;\theta)}{\partial \theta}\right)^{\prime}\right]\right|_{\theta=\theta_0}$ is positive definite.
\end{ass}

We can take any choice of measure $\pi$ such that (ii) is fulfilled. However, for some choices, numerical integration in \eqref{eq:empoperatorK} may be easier to perform. For example, we can use the normal distribution. Instead, \cite{Carrasco2000} considered the Hilbert space $L^2([0,T])$ of real-valued square-integrable functions on $[0,T]$ with $T>0$. However, \citet[][Section 3]{Carrasco2002} discussed that all results transfer by adjusting operations in the corresponding Hilbert spaces. Therefore, in (ii) we can take $\pi$ to be the uniform distribution on $[0,T]$ and replace $L^2(\pi)$ with $L^2([0,T])$. $T$ is arbitrary as long as (ii) is satisfied. For simplicity, we chose $T=1$.

With this, the CGMM estimator satisfies the following asymptotic result.
\begin{prop}\label{prop:CGMMasymnorm}
Under Assumption \ref{ass:CGMM}, the CGMM estimator is consistent and
\begin{equation}
\label{eq:CGMMasymnorm}
n^{1/2}(\hat{\theta}_{n,CGMM}(\gamma_n)-\theta_0)\stackrel{\mathcal{L}}{\rightarrow}N(0,I_{\theta_0}^{-1}),
\end{equation}
as $n\to\infty$, $\gamma_nn^{1/2}\to\infty$ and $\gamma_n\to0$, where $I_{\theta_0}^{-1}$ denotes the inverse of the Fisher information matrix (which exists due to Assumption \ref{ass:CGMM}(vi)). 
\end{prop}

In practice, we require a reasonable choice of the regularization parameter $\gamma_n$. \cite{Carrasco2017} derived the optimal estimator for $\gamma_n$. However, we choose for simplicity an ad-hoc method for selecting the regularization parameter by simply using a fixed $\gamma_n=0.01$ throughout. We justify this because \cite{Carrasco2002} found that the specific choice of the regularization parameter does not have a striking impact on the estimation precision in their simulations for the stable distribution. See Section \ref{sec:MCstudy} for more discussions about the regularization in practice.

\subsection{Cumulant matching}\label{subsec:GCM}
We also discuss a method of cumulants approach which follows \cite{Kuchler2013}. They match empirical cumulants with their theoretical counterparts. We extend this by using \citeauthor{Hansen1982}'s (\citeyear{Hansen1982}) GMM framework. We call the approach generalized method of cumulants (GMC) to distinguish it from the GMM method using characteristic function moment conditions. However, it fits well into \citeauthor{Hansen1982}'s (\citeyear{Hansen1982}) framework allowing for standard asymptotic theory. This is because we can rewrite cumulant conditions as moment conditions by using the well known relation between cumulants and moments. We start by formulating the problem as a method of moments. Let
\begin{equation}
\label{eq:theorcum}
\mathbb{E}_{\theta_0}\left[g(X;\theta)\right]=0
\end{equation}
denote the theoretical moment conditions and
\begin{equation}
\label{eq:empcum}
\frac{1}{n}\sum_{j=1}^ng(X_j;\theta)=0
\end{equation}
the empirical moment conditions. We build the function $g$ (which belong to moment and not cumulant conditions) by using the following relation between moments and cumulants
\begin{align}
\mathbb{E}[X]&=\kappa_1,\\
\mathbb{E}[X^2]&=\kappa_2+\kappa_1^2,\\
\mathbb{E}[X^3]&=\kappa_3+3\kappa_2\kappa_1+\kappa_1^3,\\
\mathbb{E}[X^4]&=\kappa_4+4\kappa_3\kappa_1+3\kappa_2^2+6\kappa_2\kappa_1^2+\kappa_1^4,\\
&\vdots\\
\mathbb{E}[X^p]&=\sum_{m=1}^p B_{p,m}(\kappa_1,\ldots,\kappa_{p-m+1}),
\end{align}
where $B_{p,m}$ denote incomplete Bell polynomials. In particular for $p$ moment conditions, we choose $g(X;\theta)=(g_1,\ldots,g_p)^{\prime}$ to be
\begin{align}
g_1&=X-\kappa_1,\\
g_2&=X^2-\kappa_2-\kappa_1^2,\\
&\vdots\label{eq:functiong}\\
g_p&=X^p-\sum_{m=1}^p B_{p,m}(\kappa_1,\ldots,\kappa_{p-m+1}).
\end{align}
Here, the theoretical cumulants $\kappa_m$ for the TSS distribution are given in \eqref{eq:cumsTSS}, and for the CTS distribution, given in \eqref{eq:cumsCTS}. For the NTS distribution, cumulants do not have an easy-to-use pattern which is why we do not use this method for the NTS distribution. The asymptotic result then is the following proposition. The proof can be found in \citet{Newey1994}.

\begin{prop}\label{prop:GMMasymnorm}
Under the assumptions of Theorems 2.6 and 3.4 in \cite{Newey1994}, the GMC estimator
\begin{equation}\label{eq:GMC}
\hat{\theta}_{n,GMC}=\argmin_{\theta\in\Theta} \left(\frac{1}{n}\sum_{j=1}^ng(X_j;\theta)\right)^{\prime}\hat{W}\left(\frac{1}{n}\sum_{j=1}^ng(X_j;\theta)\right),
\end{equation}
where $W$ is a positive semi-definite weighting matrix and an estimator $\hat{W}$ with $\hat{W}\stackrel{p}{\rightarrow}W$, is consistent for $\theta_0$. Moreover,
\begin{equation}
\label{eq:GMMasymnorm}
n^{1/2}(\hat{\theta}_{n,GMC}-\theta_0)\stackrel{\mathcal{L}}{\rightarrow}N(0,(G^{\prime}WG)^{-1}G^{\prime}W\Omega W^{\prime}G(G^{\prime}WG)^{-1}),
\end{equation}
with $G:=\left.\mathbb{E}_{\theta_0}\left[\frac{\partial g(X;\theta)}{\partial\theta}\right]\right|_{\theta=\theta_0}$, $\Omega:=\mathbb{E}_{\theta_0}\left[g(X;\theta)g(X;\theta)^{\prime}\right]$.
\end{prop}

In principle, $W$ could be any positive semi-definite weighting matrix. From standard GMM theory, we however know that taking $W=\Omega^{-1}$ and similarly $\hat{W}=\hat{\Omega}^{-1}$ yields the most efficient GMM estimator \cite[see, e.g.,][]{Hansen1982}, as long as $\Omega$ is invertible, where $\Omega=\mathbb{E}_{\theta_0}\left[g(X;\theta)g(X;\theta)^{\prime}\right]$ as above. Invertibility is not necessarily given for all possible data sets for the GMC estimator. However, it typically holds in practice. To avoid invertibility we use a regularization as in Section \ref{subsec:GMM}. Under the choice $W=\Omega^{-1}$, the asymptotic variance simplifies to $(G^{\prime}\Omega^{-1}G)^{-1}$.

%In the next section, we will state all asymptotic results which we prove in the appendix.

\section{Asymptotic results}\label{sec:results}

We present our asymptotic results for the MLE, the CGMM, and the GMC estimation method. The GMM method of \cite{Feuerverger1981} has already been discussed to have problems with the singularity of the asymptotic covariance matrix. All proofs can be found in Appendix \ref{app:proofs}. We start with a theorem for the MLE. 
\begin{thm}\label{thm:ML}
Fix any $0<\varepsilon<M<\infty$. The MLE $\hat{\theta}_{n,ML}$ for $\theta_0\in\mathrm{int}(\Theta)$ of
\begin{enumerate}
\item[(a)] the $TSS(\alpha,\delta,\lambda)$ distribution with $\theta=(\alpha,\delta,\lambda)\in\Theta=[\varepsilon,1-\varepsilon]\times[\varepsilon,M]^2$,
%\item[(b)] the $TS'(\alpha,\delta,\lambda)$ distribution with $\theta=(\alpha,\delta,\lambda)\in\Theta=\left([\varepsilon,1-\varepsilon]\cup[1+\varepsilon,2-\varepsilon]\right)\times[\varepsilon,M]^2$
\item[(b)] the $CTS(\alpha,\delta_+,\delta_-,\lambda_+,\lambda_-,\mu)$ distribution with $\theta=(\alpha,\delta_+,\delta_-,\lambda_+,\lambda_-,\mu)\in\Theta=[\varepsilon,2-\varepsilon]\times[\varepsilon,M]^4\times[-M,M]$,
\item[(c)] the $NTS(\alpha,\beta,\delta,\lambda,\mu)$ distribution with $\theta=(\alpha,\beta,\delta,\lambda,\mu)\in\Theta=[\varepsilon,1-\varepsilon]\times[-M,M]\times[\varepsilon,M]^2\times[-M,M]$,
\end{enumerate}
is consistent, asymptotically normal and asymptotically efficient as $n\to\infty$.
\end{thm}
We introduce the numbers $\varepsilon$ and $0$ to ensure that the parameter space $\Theta$ is compact. \citet[][Section 3.2 \& 3.3]{Grabchak2016} has already established strong consistency. Therefore, it only remains to show asymptotic normality and efficiency.

We next show that the CGMM also possesses the desired asymptotic properties for our tempered stable distributions.
\begin{thm}\label{thm:CGMM}
The CGMM estimator $\hat{\theta}_{n,CGMM}(\gamma_n)$ for $\theta_0\in\mathrm{int}(\Theta)$ of
\begin{enumerate}
\item[(a)] the $TSS(\alpha,\delta,\lambda)$ distribution with $\theta=(\alpha,\delta,\lambda)\in\Theta=[\varepsilon,1-\varepsilon]\times[\varepsilon,M]^2$,
%\item[(b)] the $TS'(\alpha,\delta,\lambda)$ distribution with $\theta=(\alpha,\delta,\lambda)\in\Theta=\left([\varepsilon,1-\varepsilon]\cup[1+\varepsilon,2-\varepsilon]\right)\times[\varepsilon,M]^2$
\item[(b)] the $CTS(\alpha,\delta_+,\delta_-,\lambda_+,\lambda_-,\mu)$ distribution with $\theta=(\alpha,\delta_+,\delta_-,\lambda_+,\lambda_-,\mu)\in\Theta=[\varepsilon,2-\varepsilon]\times[\varepsilon,M]^4\times[-M,M]$,
\item[(c)] the $NTS(\alpha,\beta,\delta,\lambda,\mu)$ distribution with $\theta=(\alpha,\beta,\delta,\lambda,\mu)\in\Theta=[\varepsilon,1-\varepsilon]\times[-M,M]\times[\varepsilon,M]^2\times[-M,M]$,
\end{enumerate}
is consistent, asymptotically normal and asymptotically efficient as $n\to\infty$ and $\gamma_nn^{1/2}\to\infty$ and $\gamma_n\to0$.
\end{thm}

As a side product, in the proofs of the parts (a) and (b) of the Theorems \ref{thm:ML} and \ref{thm:CGMM} we also verify that the assumptions hold for the TS' distribution hold. However, asymptotic normality does not hold for the full range of $\alpha\in(0,2)$ but only for $\alpha\ge1$. The reason is, as already mentioned in Section \ref{subsec:PTS}, that the support of the density function of the TS' distribution depends on $\theta$ for $\alpha\in(0,1)$. We summarize the result in the following corollary. 

\begin{cor}\label{cor:PTS}
The ML and the CGMM estimators for the $TS^{\prime}(\alpha,\delta,\lambda)$ with $(\alpha,\delta,\lambda)\in\Theta=[1,2-\varepsilon]\times[\varepsilon,M]^2$ is consistent, asymptotically normal and asymptotically efficient as $n\to\infty$ and $\gamma_nn^{1/2}\to\infty$ and $\gamma_n\to0$.
\end{cor}

\citet[][Lemma 6.1]{Kuchler2013} showed that the method of cumulant matching is locally identified for $\alpha\in(0,1)$, i.e., the equations of the moment conditions do have a root which is unique in an open neighborhood around $\theta_0$ by the implicit function theorem. Moreover, they showed in Proposition 5.4 that the TSS is globally identified and thus consistent. We show local identification for the GMC.% However, the assumptions of Proposition \ref{prop:GMMasymnorm} require that the root is unique for the whole domain. We show that this does not hold by constructing counterexamples for a global inverse function theorem. We do this for the TSS distribution. The CTS distribution works analogously.

\begin{thm}\label{thm:GMC}
Consider the GMC estimator $\hat{\theta}_{n,GMC}$ for $\theta_0\in\mathrm{int}(\Theta)$. Then
\begin{enumerate}
\item[(a)] the $TSS(\alpha,\delta,\lambda)$ distribution with $\theta=(\alpha,\delta,\lambda)\in\Theta=[\varepsilon,1-\varepsilon]\times[\varepsilon,M]^2$,
\item[(b)] the $CTS(\alpha,\delta_+,\delta_-,\lambda_+,\lambda_-,\mu)$ distribution with $\theta=(\alpha,\delta_+,\delta_-,\lambda_+,\lambda_-,\mu)\in\Theta=[\varepsilon,2-\varepsilon]\times[\varepsilon,M]^4\times[-M,M]$,
%\item[(b)] the cumulant matching of \cite{Kuchler2013} and more generally $\hat{\theta}_{n,GMC}$ are not globally consistent.
\end{enumerate} 
is locally identified.
\end{thm}

Local identification is necessary but not sufficient for the consistency, see Section 2.2.3 of \cite{Newey1994}.

%This means it might be the case that, depending on the true $\theta_0$, more than one root exists.% However, these counterexamples are somewhat artificial, see Remark \ref{rem:globalhomeo} in the Appendix.

\section{Monte Carlo study}\label{sec:MCstudy}

In this section, we compare empirical properties of the proposed estimators in a simulation study. In order to do so, we simulate $n=100$ and $n=1000$ random numbers distributed according to the distributions $TSS(0.5, 1,1)$, $CTS(1.5, 1,1,1,1,0)$, and $NTS(0.5,0,1,1,0)$. We estimate the parameters using the MLE, the GMM method according to \cite{Feuerverger1981}, and the CGMM method according to \cite{Carrasco2017}. For the TSS and CTS distributions, we additionally compute the GMC estimator using three different numbers of moment conditions, i.e., the just-identified case (3 moment conditions for TSS, 6 moment conditions for CTS) and two overidentified cases (4 and 5 moment conditions for TSS, 7 and 8 moment conditions for CTS). For all optimization problems we use the \texttt{optim} function in R with the L-BFGS-B method by \cite{Byrd1995}, which is suitable for box-constraint optimization. We evaluate the density functions with the FFT method for the CTS and NTS distributions using the \texttt{fft} function in R. For the TSS distribution, we employ the relation with the stable distribution \eqref{eq:dTSS} and the \texttt{stabledist} package \cite[]{stabledist} package. We use the uniform distribution on $(0,1)$ for $\pi$.\footnote{As explained above, the uniform distribution does not fulfill Assumption \ref{ass:CGMM}.(ii) but the assumptions of \cite{Carrasco2000}. For comparison, we therefore have also used the normal distribution, but the results are robust to this choice and hence not reported.} For the GMC estimator, we compute a first step estimator $\hat{\theta}_{(1)}$ using the identity matrix as the weighting matrix. With this, we compute $\hat{\Omega}=n^{-1}\sum_{j=1}^n g(X_j;\hat{\theta}_{(1)})g(X_j;\hat{\theta}_{(1)})^{\prime}$ and next the asymptotically efficient GMC estimator with the weighting matrix $\hat{\Omega}^{-1}$. As mentioned in Section \ref{subsec:GMM}, we use the Tikhonov regularization \eqref{eq:operatorKest} for the inversion of $K$ to compute the CGMM estimator. To avoid problems with numerical inversion of the GMM and the GMC estimators, we also use a regularization to invert the respective matrices $\hat\Omega$. We choose a cut-off regularization \cite[see][]{Carrasco2000} for GMM\footnote{We also tried the LF and the Tikhonov regularization with no qualitative difference.} and a Tikhonov regularization \eqref{eq:operatorKest} for GMC. For all setups, we choose a regularization parameter of 0.01. However, we note that for the regularization parameter fine-tuning is possible to obtain more precise estimation results. \cite{Juessen2023}, in his Master's thesis, reports estimation results for $\gamma_n=0.1$, finding an improvement of the accuracy. For the GMM, we use an equally spaced grid with grid size $R=10$ (for TSS and NTS) or $R=20$ for (CTS). The grid is determined as in \cite{Kharrat2016} by taking $\varepsilon$ as the smallest and the first root of the real part of the empirical characteristic function as the largest values of the grid. Other suggestions for grids can be found in \cite{Kharrat2016}. We set $\varepsilon=10^{-6}$. Although $M$ theoretically needs to be finite, we made good experiences with the choice of $\verb|Inf|$ as upper bounds in R. The implemented routines are in the \texttt{TempStable} R package \cite[]{TempStable}. We repeat the experiments in 10,000 independent Monte Carlo replications.

Table \ref{tab:TSS} shows empirical bias and the empirical root mean squared errors (RMSE) (in parenthesis) for the TSS distribution for each of the parameters. In the last column, we display the average runtime for estimation in seconds. The striking difference is that while the GMM and the GMC estimation perform in less than a second MLE and CGMM need considerably more time to compute. As expected, the estimates are more precise for a larger sample. Importantly, all methods (except the MLE) suffer from a small sample size which implies that the algorithms run into boundary solutions, i.e., finding stability parameters close to zero (boundary solutions for the other parameters rarely occur). This implies a negative empirical bias for $\alpha$. The MLE outperforms the other methods followed by CGMM. The GMC with 4 moment conditions also works fairly well, adding further moment conditions has no additional value.

\begin{table}[htbp]
  \centering
    \begin{tabular}{lrrrrr}
         \toprule 
					& $n$ & $\alpha$ & $\delta$ & $\lambda$ & time \\
    \midrule
		CGMM  & 100   & -0.071 & 0.458 & 0.206 & 43 \\
          &       & (0.22)  & (1.041) & (0.594) &  \\
          & 1000  & -0.011 & 0.057 & 0.031 & 1256 \\
          &       & (0.073) & (0.262) & (0.192) &  \\
    GMC ($p=3$) & 100   & -0.169 & 0.982 & 0.448 & \textbf{0.2} \\
          &       & (0.282) & (1.637) & (0.796) &  \\
          & 1000  & -0.028 & 0.127 & 0.072 & \textbf{0.6} \\
          &       & (0.096) & (0.358) & (0.243) &  \\
    GMC ($p=4$)  & 100   & -0.142 & 0.752 & 0.345 & 0.4 \\
          &       & (0.27)  & (1.388) & (0.687) &  \\
          & 1000  & -0.002 & 0.02  & \textbf{0.005} & 0.8 \\
          &       & (0.068) & (0.226) & (0.177) &  \\
    GMC ($p=5$) & 100   & -0.037 & 0.489 & 0.489 & 0.5 \\
          &       & (0.292) & (1.359) & (0.742) &  \\
          & 1000  & -0.033 & 0.139 & 0.099 & 1 \\
          &       & (0.089) & (0.347) & (0.215) &  \\
    GMM & 100   & -0.103 & 0.706 & 0.292 & 0.6 \\
          &       & (0.261) & (1.488) & (0.749) &  \\
          & 1000  & -0.013 & 0.063 & 0.034 & 0.6 \\
          &       & (0.077) & (0.275) & (0.198) &  \\
    MLE & 100   & \textbf{-0.004} & \textbf{0.098} & \textbf{0.056} & 82 \\
          &       & (\textbf{0.117}) & (\textbf{0.518}) & (\textbf{0.409}) &  \\
          & 1000  & \textbf{-0.001} & \textbf{0.013} & 0.01  & 802 \\
          &       & (\textbf{0.038}) & (\textbf{0.137}) & (\textbf{0.125}) &  \\
    \bottomrule
		\end{tabular}%
		  \caption{Empirical bias and RMSE (in parenthesis) for the parameters of the TSS distribution for different estimation methods. Last column shows average runtime in seconds. Smallest values in bold.}
  \label{tab:TSS}%
\end{table}%

Table \ref{tab:CTS} shows similar experimental results for the CTS distribution. Again, we report the empirical bias and the RMSE for each of the parameters. As expected, estimating 6 parameters is more demanding than estimating 3 as above. We see that for the GMC and GMM methods extremely high values of bias and RMSEs occur, which is due to rare extremely high parameter estimates. Therefore, we also compute the median absolute deviation (MAD) from the true parameters, printed in square brackets. Generally, we observe that all methods fail to provide good estimates for 100 observations. In this case, the optimization algorithms find boundary solutions for many of the randomly drawn sets. Thus, the CTS distribution should only be used as a model if the sample size is not too small. The GMC and GMM methods also fail for 1000 observations for most of the Monte Carlo replications. For 1000 observations, the MLE performs better than the latter methods. However, it still has difficulties to estimate the stability index correctly in some instances. The CGMM method has the longest runtime but works fairly well especially compared with the other estimators. Boundary estimates rarely occur.

\begin{table}[htbp]
  \centering
    \begin{tabular}{lrrrrrrrr}
                   \toprule 
					& $n$ & $\alpha$ & $\delta_+$ & $\delta_-$ & $\lambda_+$ & $\lambda_-$ & $\mu$ & time \\
					\midrule
    CGMM  & 100   & -0.518 & \textbf{2.207} & \textbf{2.377} & 2.918 & 2.76  & 0.001 & 143 \\
          &       & (0.778) & (\textbf{6.096}) & (\textbf{6.369}) & (\textbf{6.398}) & (\textbf{6.091}) & (0.184) &  \\
          &       & [\textbf{0.427}] & [0.999999] & [0.999999] & [\textbf{1.114}] & [1.117] & [0.124] &  \\
          & 1000  & \textbf{-0.05} & \textbf{0.607} & \textbf{0.671} & 1.841 & 1.781 & 0     & 5443 \\
          &       & (\textbf{0.157}) & (\textbf{2.028}) & (\textbf{2.01})  & (3.949) & (3.88)  & (0.059) &  \\
          &       & [\textbf{0.074}] & [\textbf{0.788}] & [\textbf{0.8}]   & [\textbf{0.626}] & [\textbf{0.624}] & [0.039] &  \\
GMC (p=6) & 100   & -0.903 & $1.48\mathrm{E}5$ & $1.14\mathrm{E}5$ & $1.55\mathrm{E}8$ & $2.12\mathrm{E}8$ & \textbf{0}     & \textbf{1} \\
          &       & (1.206) & ($4.92\mathrm{E}6$) & ($4.96\mathrm{E}6$) & ($9.84\mathrm{E}9$) & ($2\mathrm{E}10$) & (0.221) &  \\
          &       & [1.49999] & [1.432] & [1.571] & [1.825] & [1.82]  & [0.145] &  \\
          & 1000  & -0.865 & $2.79\mathrm{E}5$ & $6.11\mathrm{E}5$ & $3.16\mathrm{E}8$ & $7.82\mathrm{E}8$ & 0     & \textbf{3} \\
          &       & (1.113) & ($9.47\mathrm{E}6$) & ($2.34\mathrm{E}7$) & ($2.17\mathrm{E}7$) & ($7.4\mathrm{E}4$) & (0.06)  &  \\
          &       & [1.175] & [3.521] & [3.959] & [1.584] & [1.636] & [0.039] &  \\
GMC (p=7) & 100   & \textbf{0.281} & 275.6 & 361.6 & $2.65\mathrm{E}5$ & 5918  & -0.001 & 3 \\
          &       & (\textbf{0.563}) & (9876)  & ($1.24\mathrm{E}4$) & ($2.5\mathrm{E}7$) & ($2.86\mathrm{E}5$) & (0.482) &  \\
          &       & [0.499] & [\textbf{0.9995}] & [\textbf{0.99995}] & [0.92]  & [0.918] & [0.166] &  \\
          & 1000  & 0.158 & $1.5\mathrm{E}4$ & $1.26\mathrm{E}5$ & $5.27\mathrm{E}8$ & $9.41\mathrm{E}7$ & 0.001 & 4 \\
          &       & (0.586) & ($8.9\mathrm{E}5$) & ($1.1\mathrm{E}7$) & ($3.4\mathrm{E}10$) & ($5.9\mathrm{E}9$) & (0.07)  &  \\
          &       & [0.479] & [0.99998] & [0.99998] & [0.888] & [0.861] & [0.043] &  \\
GMC (p=8) & 100   & -0.313 & 3044  & 2302  & $3.18\mathrm{E}5$ & $1.14\mathrm{E}6$ & 0.002 & 2 \\
          &       & (0.887) & ($2.13\mathrm{E}5$) & ($1.6\mathrm{E}5$) & ($1.67\mathrm{E}7$) & ($7.1\mathrm{E}7$) & (\textbf{0.109}) &  \\
          &       & [0.499995] & [0.99999] & [0.99999] & [0.964] & [0.987] & [0.19]  &  \\
          & 1000  & -0.26 & 1425  & 9768  & $2.47\mathrm{E}5$ & $5.63\mathrm{E}5$ & -0.001 & 4 \\
          &       & (0.83)  & ($9.04\mathrm{E}4$) & ($8.61\mathrm{E}5$) & ($2.1\mathrm{E}7$) & ($5.1\mathrm{E}7$) & (0.208) &  \\
          &       & [0.499998] & [0.999999] & [0.999999] & [1.135] & [1.202] & [0.051] &  \\
GMM			  & 100   & -1.207 & 19.7  & 20.31 & 18.3  & 17.7  & -0.001 & 5 \\
          &       & (1.318) & (56.63) & (59.25) & (120.9) & (92.04) & (0.256) &  \\
          &       & [1.499999] & [1.91]  & [2.129] & [2.263] & [2.2]   & [0.128] &  \\
          & 1000  & -0.802 & 13.16 & 13.57 & 7.828 & 8.015 & 0     & 13 \\
          &       & (1.021) & (28.13) & (28.46) & (23.88) & (24.91) & (0.074) &  \\
          &       & [0.816] & [3.755] & [4.222] & [2.724] & [2.728] & [0.043] &  \\
   MLE		  & 100   & -0.856 & 5.69  & 6.02  & \textbf{1.695} & \textbf{2.204} & 0.002 & 390 \\
          &       & (1.131) & (13.01) & (13.94) & (10.75) & (18.71) & (0.183) &  \\
          &       & [1.1239] & [0.999999] & [0.999999] & [10.013] & [\textbf{1.045}] & [\textbf{0.123}] &  \\
          & 1000  & -0.292 & 2.6   & 2.775 & \textbf{0.809} & \textbf{0.842} & \textbf{0}     & 3919 \\
          &       & (0.716) & (6.638) & (6.877) & (\textbf{1.382}) & (\textbf{1.409}) & (\textbf{0.058}) &  \\
          &       & [0.399] & [0.97]  & [0.975] & [0.819] & [0.827] & [\textbf{0.039}] &  \\
        \bottomrule
		\end{tabular}%
    \caption{Empirical bias, RMSE (in parenthesis) and MAD [in square brackets] for the parameters of the CTS distribution for different estimation methods. Last column shows average runtime in seconds. Smallest values in bold.}
	\label{tab:CTS}%
\end{table}%

Table \ref{tab:NTS} presents the results for the NTS distribution. For this distribution, we only use the CGMM, the GMM, and the ML methods. As before, a larger sample is beneficial since small samples lead to boundary estimates. For example, the bias and RMSE for $\beta$ are very large for the GMM method because of large outliers. As for the CTS above, the CGMM method seems to perform better than MLE and much better than the GMM method, which performs poorly even for a larger sample.

\begin{table}[htbp]
  \centering
    \begin{tabular}{lrrrrrrr}
          \toprule
					& $n$ & $\alpha$ & $\beta$ & $\delta$ & $\lambda$ & $\mu$ & time \\
    \midrule
		CGMM & 100   & 0.084 & -0.033 & \textbf{0.128} & 2.515 & 0.044 & 75 \\
          &       & (\textbf{0.278}) & (\textbf{2.535}) & (\textbf{0.938}) & (\textbf{5.641}) & (3.146) &  \\
          & 1000  & \textbf{0.025} & \textbf{0.001} & \textbf{-0.039} & \textbf{0.081} & \textbf{-0.0004} & 1653 \\
          &       & (\textbf{0.121}) & (0.21)  & (\textbf{0.375}) & (\textbf{0.73})  & (0.333) &  \\
    GMM & 100   & -0.063 & $-2.38\mathrm{E}6$ & 1.891 & 15.606 & 0.231 & \textbf{3} \\
          &       & (0.412) & ($2.38\mathrm{E}8$) & (19.062) & (169.6) & (8.757) &  \\
          & 1000  & -0.14 & -0.002 & 1.572 & 0.972 & -0.037 & \textbf{5} \\
          &       & (0.417) & (5.44)  & (2.985) & (3.012) & (6.719) &  \\
    MLE & 100   & \textbf{0.011} & \textbf{-0.007} & 0.334 & \textbf{1.873} & \textbf{0.021} & 185 \\
          &       & (0.403) & (3.952) & (3.716) & (7.831) & (\textbf{1.744}) &  \\
          & 1000  & -0.105 & 0.002 & 1.0377 & 0.129 & -0.003 & 2063 \\
          &       & (0.366) & (\textbf{0.172}) & (2.101) & (1.173) & (\textbf{0.279}) &  \\
    \bottomrule
		\end{tabular}%
		  \caption{Empirical bias and RMSE (in parenthesis) for the parameters of the NTS distribution for different estimation methods. Last column shows average runtime in seconds. Smallest values in bold.}
  \label{tab:NTS}%
\end{table}%

Next, we test the asymptotic normality of the parameters by computing the coverage of asymptotic confidence intervals based on the inverse of the Fisher information matrix. Unfortunately, the Fisher information matrix is cumbersome to compute for the CTS and NTS distributions because they rely on several numerical algorithms (FFT, numerical derivation, and numerical integration) such that the result is numerically unstable. For the TSS distribution, there exists the relation to the stable distribution \eqref{eq:dTSS} and series representations like \eqref{eq:dTSSseries}. Thus, we restrict ourselves to the analysis for the TSS distribution. Table \ref{tab:TSScov} reports the coverage (in \%) of asymptotic confidence intervals with a nominal confidence level of 95\% for the three different parameters. To compute these, we have used the very same setup and parameter estimation results as for Table \ref{tab:TSS} above. As for the comparison of the bias and RMSE, the MLE outperforms the other methods but this time very clearly. The reason is that the bias is decisively smaller for the MLE method so that for the other methods the point estimates and thus the confidence intervals are simply too far away from the true values. We note that because we use asymptotic confidence intervals the widths for the different methods are similar. This leads to a smaller coverage which is far from the nominal level and we do not observe an improvement with increasing sample size in all of the cases. On the other hand, the MLE coverage rate is close to the nominal confidence level for large sample sizes. It is not fully clear why there is a higher estimation variance and bias of the CGMM than of the MLE for the TSS distribution. One explanation might be due to the nature of the TSS distribution which is only on $\mathbb{R}_+$. It is likely that more fine-tuning for the many estimation parameters is needed to obtain more precise results. Especially, there may be room for improvement by using an optimal regularization parameter instead of the ad-hoc rule.

\begin{table}[htbp]
  \centering
    \begin{tabular}{lrrrr}
         \toprule 
					& $n$ & $\alpha$ & $\delta$ & $\lambda$  \\
    \midrule
		CGMM  & 100   & 68.19 & 67.54 & 72.64 \\
          & 1000  & 62.59 & 68.94 & 73.71 \\
    GMC ($p=3$) & 100   & 63.12 & 62.71& 65.7 \\
          & 1000  & 52.56 & 59.68 & 64.95 \\
    GMC ($p=4$)  & 100   & 63.35 & 63.25 & 70.51 \\
          & 1000  & 66.95 & 71.84  & 76.79 \\
    GMC ($p=5$) & 100   & 66.4 & 67.58 & 71.86 \\
          & 1000  & 65.37 & 73.7 & 80.29 \\
    GMM & 100   & 59.3&58.64&63.77 \\
          & 1000  &  60.66&67.35&72.88\\
    MLE & 100   &  \textbf{83.28} & \textbf{81.34} & \textbf{85.94}  \\
          & 1000  &  \textbf{88.6} & \textbf{89.07} & \textbf{89.45}  \\
    \bottomrule
		\end{tabular}%
		  \caption{Confidence interval coverage (in \%) for the parameters of the TSS distribution for different estimation methods with a confidence level of 95\%. Largest values in bold.}
  \label{tab:TSScov}%
\end{table}%

We have also tried other parameter constellations for $\theta$ without any qualitative difference except for $\alpha$ when it is close to the boundary. We omit the details here. To conclude, the CGMM works well for all three examples and is the only reliable estimator for the CTS distribution. Unfortunately, the runtime is quite lengthy as well as for the MLE, which finds boundary solutions more frequently. The GMM and GMC estimators are only reasonable for the TSS subordinator. Therefore, we recommend using the CGMM estimator for the CTS and NTS distribution.

\section{Applications}\label{sec:application}

We discuss financial applications to motivate the use of tempered stable distributions. Tempered stable distributions have already been proposed to model log-returns of financial assets, see e.g., \cite{CGMY2002,Fallahgoul2019}. In the case of electricity markets, e.g., \cite{Sabino2022} models the evolution of electricity spot prices by a tempered stable driven Ornstein-Uhlenbeck process.

Here, we analyze log-returns of three financial assets. We consider the S\&P 500 index (2012-06-01 to 2022-05-31), the German DAX index (2012-06-01 to 2022-05-31), and base-load spot prices from the German power exchange EEX (2018-09-30 to 2022-05-31). We obtain daily data from Refinitiv's Eikon. Before modeling tempered stable distributions we perform some simple manipulations and fit preliminary models to clean the data. In this order, we first deseasonalize the data of the EEX regarding its weekly profile by applying a moving average filter, see \cite{Weron2005,Weron2007}. Second, we exclude the rare cases of negative prices of the EEX. Next, we compute log-returns for the three indices (for the EEX we also omit all log-returns next to days with negative prices). Last, we fit GARCH(1,1) models (with normal errors) using the \texttt{tseries} package \cite[]{Trapletti2020} to each of the log-return series to remove stochastic volatility from the data which can mistakenly interpreted as evidence for heavy tailed distributions \cite[]{Fallahgoul2019}. We then fit the CTS, the NTS, and the stable distributions to the residuals of the GARCH model. With this approach, we follow \cite{Goode2015} who found that this quasi-MLE performs nearly identically as a correctly specified MLE of a GARCH model with CTS or NTS distributed errors.

Figures \ref{fig:SPX} to \ref{fig:EEX} show the original time series (a), the (deseasonalized where appropriate) log-returns (b), and the GARCH residuals (c). We observe in panels (b) that for each of the series of log-returns stochastic volatility is apparent. The GARCH residuals do not exhibit volatility clustering anymore. However, the residuals reveal skewness and heavy tails which is why we next fit tempered stable distributions to them.

\begin{figure*}
	\centering
  \subfloat[][Price process]{\includegraphics[width=.65\textwidth]{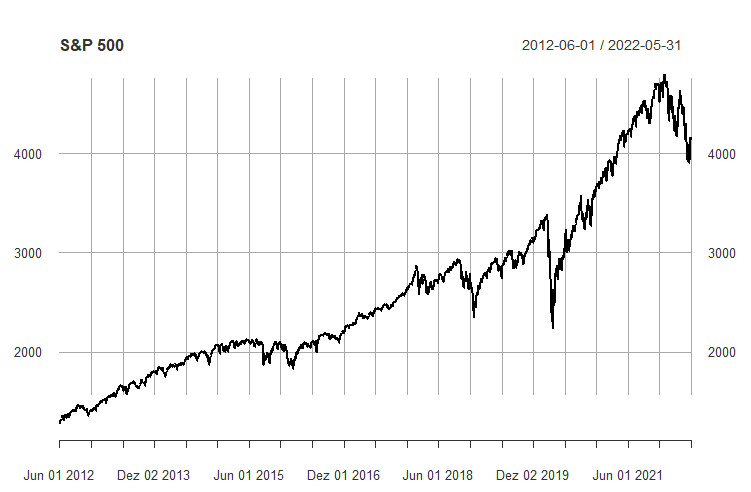}} \par
  
  \subfloat[][Log-returns]{\includegraphics[width=.65\textwidth]{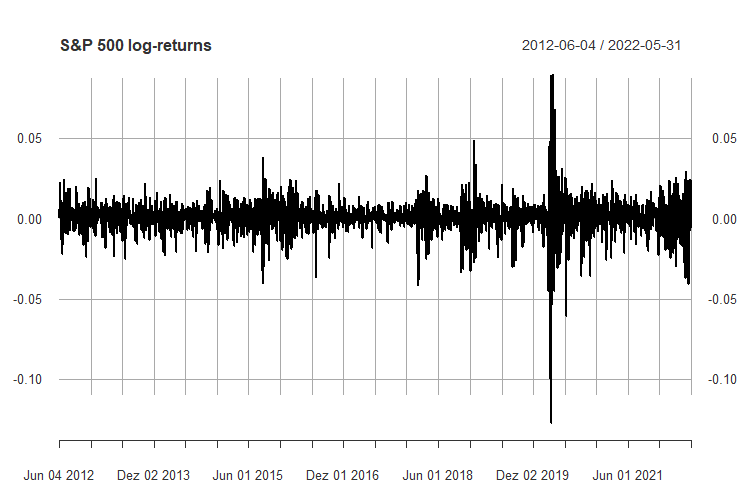}}\par

	\subfloat[][GARCH residuals]{\includegraphics[width=.65\textwidth]{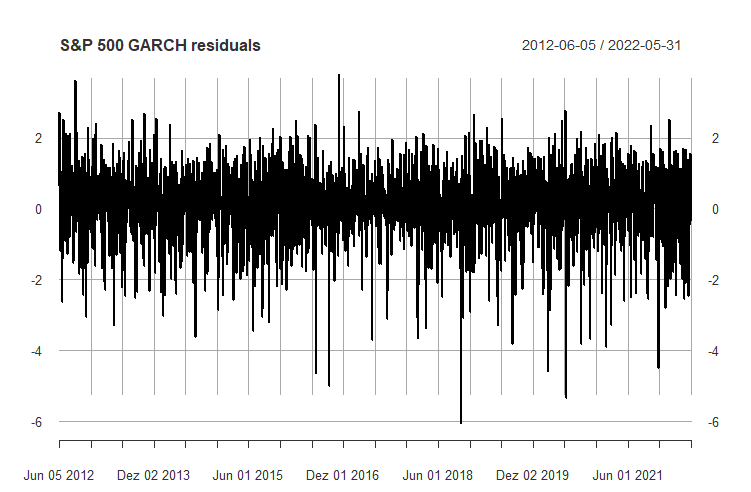}}\par
  
  \caption{Price process (a), log-returns (b) and GARCH residuals (c) for the S\&P 500 index from 2012-06-01 to 2022-05-31.}
  \label{fig:SPX}
\end{figure*}

\begin{figure*}
	\centering
  \subfloat[][Price process]{\includegraphics[width=.65\textwidth]{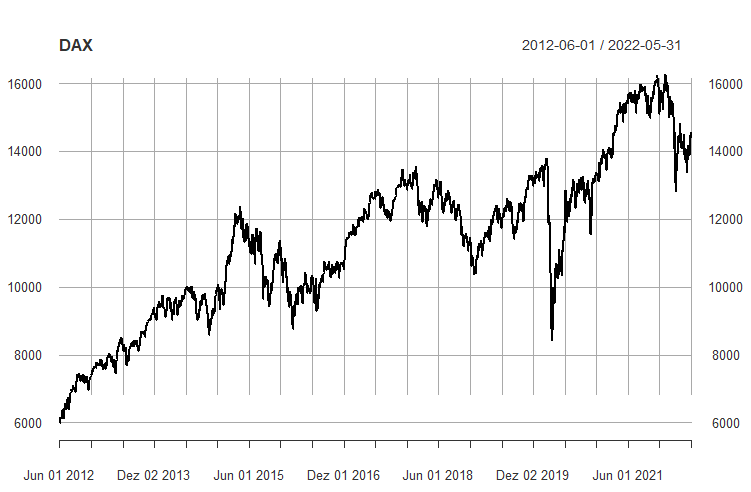}} \par
  
  \subfloat[][Log-returns]{\includegraphics[width=.65\textwidth]{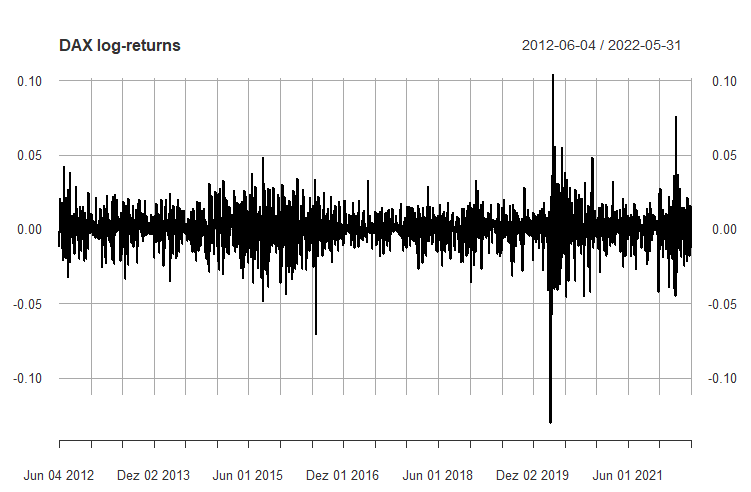}}\par

	\subfloat[][GARCH residuals]{\includegraphics[width=.65\textwidth]{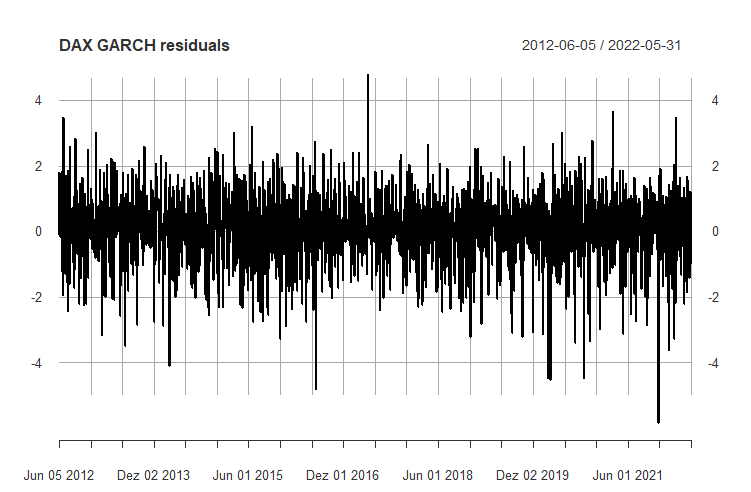}}\par
  
  \caption{Price process (a), log-returns (b) and GARCH residuals (c) for the DAX index from 2012-06-01 to 2022-05-31.}
  \label{fig:DAX}
\end{figure*}

\begin{figure*}
	\centering
  \subfloat[][Price process]{\includegraphics[width=.65\textwidth]{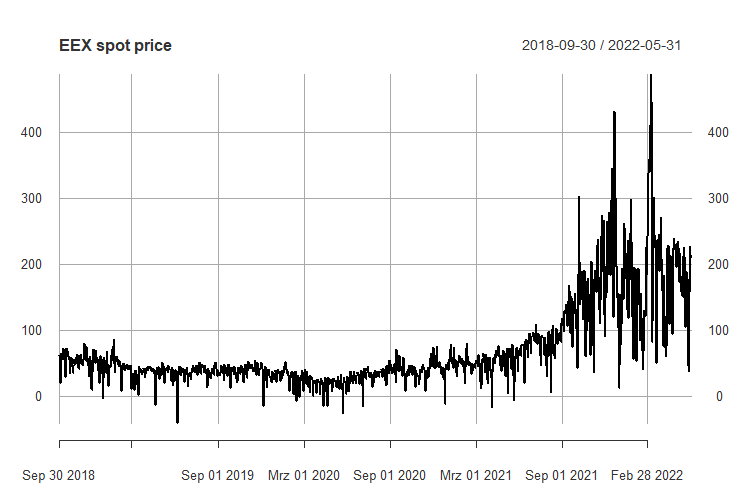}} \par
  
  \subfloat[][Log-returns]{\includegraphics[width=.65\textwidth]{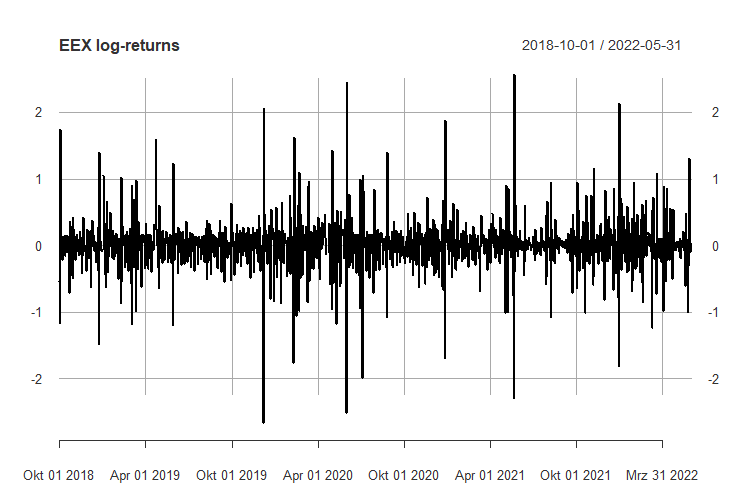}}\par

	\subfloat[][GARCH residuals]{\includegraphics[width=.65\textwidth]{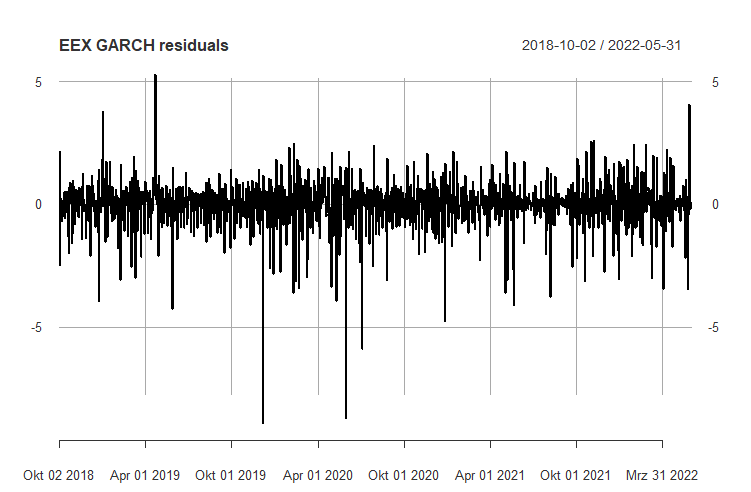}}\par
  
  \caption{Price process (a), deseasonalized log-returns (b) and GARCH residuals (c) for the EEX from 2018-09-30 to 2022-05-31.}
  \label{fig:EEX}
\end{figure*}

We compare the CTS and the NTS distributions with a univariate stable distribution as a baseline model (the TSS distribution is not considered since it only models positive data and here data can be negative). Needless to say, there are plenty of other models for log-returns in the literature, e.g., generalized hyperbolic models (and their subclasses) by \cite{barndorff1977exponentially}, or finite mixture models \cite[]{Massing2021}. For conciseness, we do not present them here but refer to the aforementioned references and \cite{Massing2019b} for a comparison. 

To compare the goodness-of-fit we use the Kolmogorov-Smirnov (KS) and the Anderson-Darling (AD) statistics. Lower statistics indicate a better fit. It is known that the KS distance better reflects the fit around the center of the distribution while the AD statistic concentrates on the tails of the distributions \cite[]{razali2011power}. Of course, the CTS and the NTS distributions have a larger number of parameters hence a better fit is to be expected. Therefore, we also compute the Akaike Information Criterion (AIC) and the Bayesian Information Criterion (BIC) to penalize large models to avoid overfitting. The penalization of larger models is higher for the BIC.

For brevity, we decide to only present estimates and plots obtained with the CGMM estimator. %The reason for this choice is that it performed best in the simulation study in Section \ref{sec:MCstudy}.
Table \ref{tab:estimates} shows parameter estimates for the three time series and the three distributions. Table \ref{tab:distances} presents KS, AD, AIC, and BIC statistics.

\begin{table}[htbp]
  \centering
    \begin{tabular}{lrrrrrrrr}
         \toprule 
					&       & $\alpha$ & $\beta$ & $\delta_+$ & $\delta_-$ & $\lambda_+$ & $\lambda_-$ & $\mu$ \\
					\midrule
    S\&P & Stable & 1.847 &       & 0     & 0.114 &       &       & 0.177 \\
      ($n=2514$)    & CTS   & 0.659 &       & 0.37  & 0.996 & 1.218 & 1.139 & 0.054 \\
          & NTS   & 0.31 & -0.382 & 0.788 &       & 1.185 &       & 0.404 \\
    DAX & Stable & 1.87  &       & 0.009 & 0.091 &       &       & 0.11 \\
      ($n=2527$)   & CTS   & 0.614 &       & 0.513 & 0.984 & 1.255 & 1.212 & 0.025 \\
          & NTS   & 0.224 & -0.231 & 0.952 &       & 1.229 &       & 0.248 \\
    EEX   & Stable & 1.713 &       & 0.023 & 0.123 &       &       & 0.107 \\
    ($n=1322$)      & CTS   & 0.369 &       & 0.514 & 0.378 & 1.24  & 0.7   & -0.032 \\
          & NTS   & 0.702 & -0.251 & 0.155 &       & 0.125 &       & 0.175 \\
					\bottomrule
    \end{tabular}%
		  \caption{Parameter estimates for the stable, the CTS and the NTS distributions fitted to S\&P 500, DAX and EEX GARCH residuals. The numbers in parenthesis are the sample sizes of the GARCH residuals.}
  \label{tab:estimates}%
\end{table}%

\begin{table}[htbp]
  \centering
    \begin{tabular}{lrrrrr}
		\toprule
          &       & KS & AD & AIC & BIC \\
					\midrule
    S\&P & Stable & \textbf{0.028} & \textbf{2.891} & 7027  & \textbf{17075} \\
          & CTS   & 0.034 & 3.41  & 6965  & 22037 \\
          & NTS   & 0.029 & 3.206 & \textbf{6941}  & 19501 \\
    DAX & Stable & 0.033 & 4.682 & 7078  & \textbf{17178} \\
          & CTS   & 0.02  & \textbf{2.321} & 6986  & 22136 \\
          & NTS   & \textbf{0.013} & 3.855 & \textbf{6973}  & 19598 \\
    EEX & Stable & 0.043 & 2.004 & 3442  & \textbf{8722} \\
          & CTS   & 0.031 & 5.135 & 3335  & 11255 \\
          & NTS   & \textbf{0.025} & \textbf{1.961} & \textbf{3322}  & 9922 \\
					\bottomrule
    \end{tabular}%
		  \caption{Goodness-of-fit statistics (lower statistics indicate a better fit) for the stable, the CTS and the NTS distributions fitted to S\&P 500, DAX and EEX GARCH residuals. Smallest values in bold.}
  \label{tab:distances}%
\end{table}%

For the KS and AD statistics, we observe a mixed pattern. For the S\&P 500 the stable distribution is favored while for the DAX and EEX the CTS or NTS distributions have lower values. While in each case the NTS distribution has the lowest AIC, the stable distribution always has the lowest BIC. To visualize goodness-of-fit we moreover depict QQ-plots.
Figures \ref{fig:SPXqq}--\ref{fig:EEXqq} plot sample quantiles versus theoretical quantiles for the different scenarios. The solid line is the reference line. We observe that, although in some cases the stable distribution has lower KS or AD statistics, the QQ-plots suggest that the tails of the stable distribution are too heavy. CTS and NTS distributions seem to provide a better fit.

\begin{figure*}
	\centering
  \subfloat[][]{\includegraphics[width=.65\textwidth]{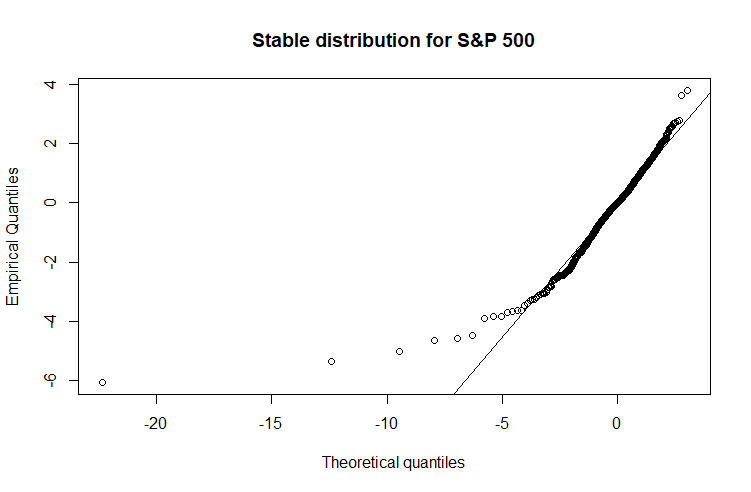}} \par
  
  \subfloat[][]{\includegraphics[width=.65\textwidth]{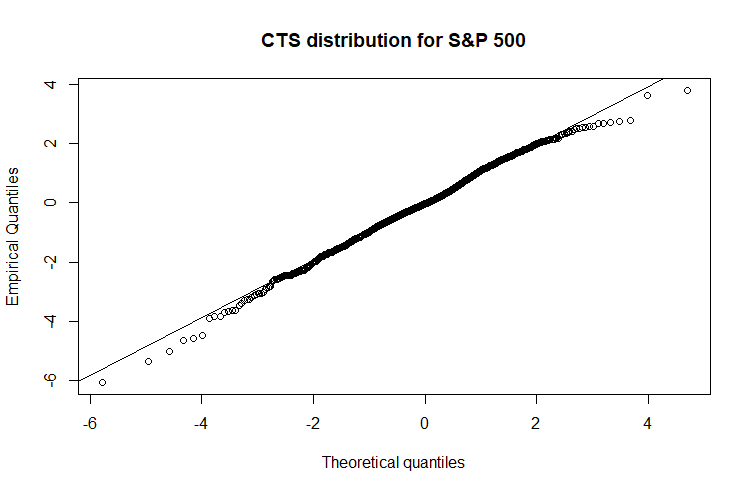}}\par

	\subfloat[][]{\includegraphics[width=.65\textwidth]{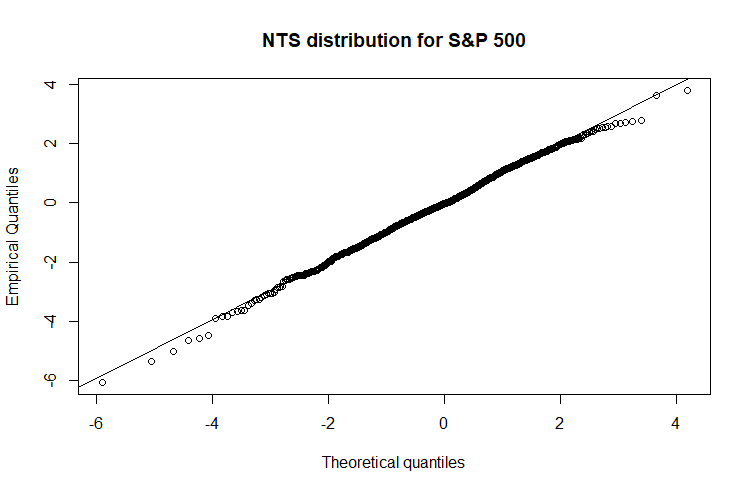}}\par
  
  \caption{QQ-Plots comparing the goodness-of-fit of the stable, the CTS and the NTS distributions for the S\&P 500. The solid line is the reference line.}
  \label{fig:SPXqq}
\end{figure*}

\begin{figure*}
	\centering
  \subfloat[][]{\includegraphics[width=.65\textwidth]{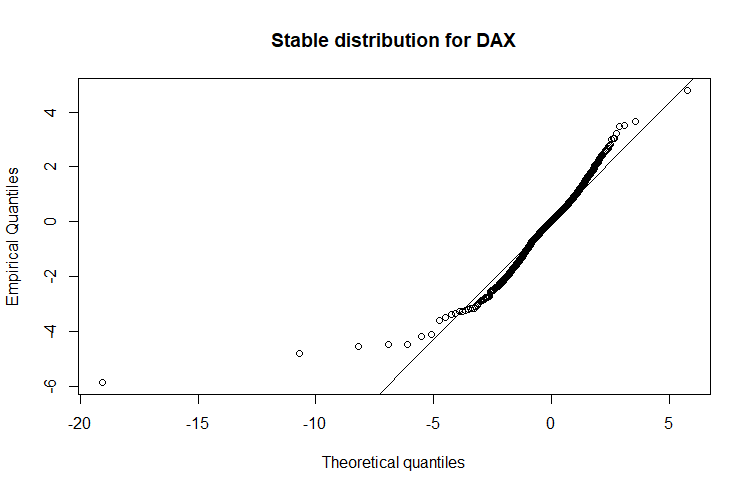}} \par
  
  \subfloat[][]{\includegraphics[width=.65\textwidth]{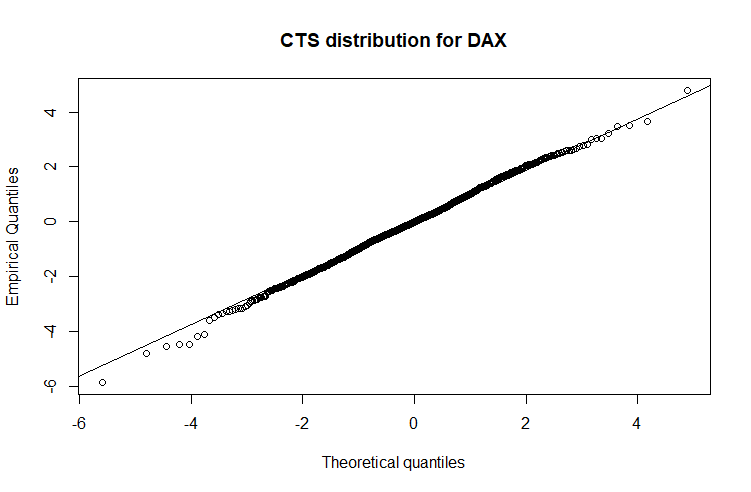}}\par

	\subfloat[][]{\includegraphics[width=.65\textwidth]{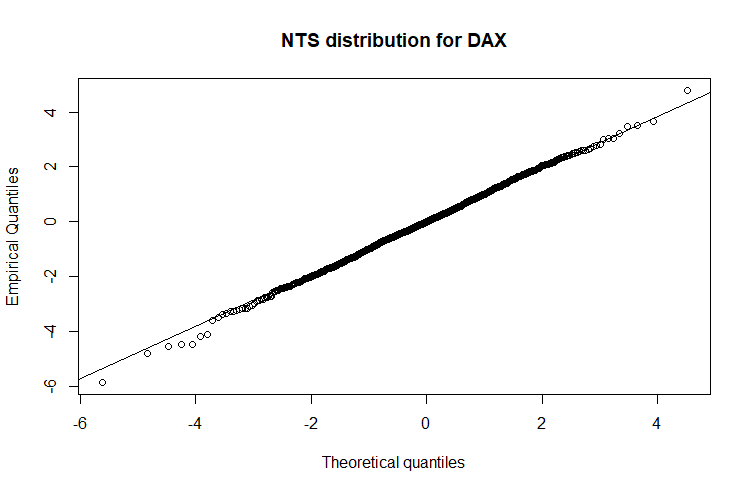}}\par
  
  \caption{QQ-Plots comparing the goodness-of-fit of the stable, the CTS and the NTS distributions for the DAX. The solid line is the reference line.}
  \label{fig:DAXqq}
\end{figure*}

\begin{figure*}
	\centering
  \subfloat[][]{\includegraphics[width=.65\textwidth]{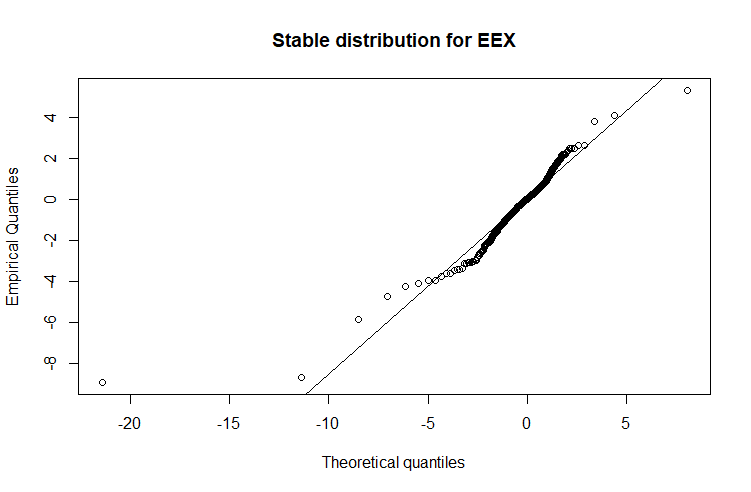}} \par
  
  \subfloat[][]{\includegraphics[width=.65\textwidth]{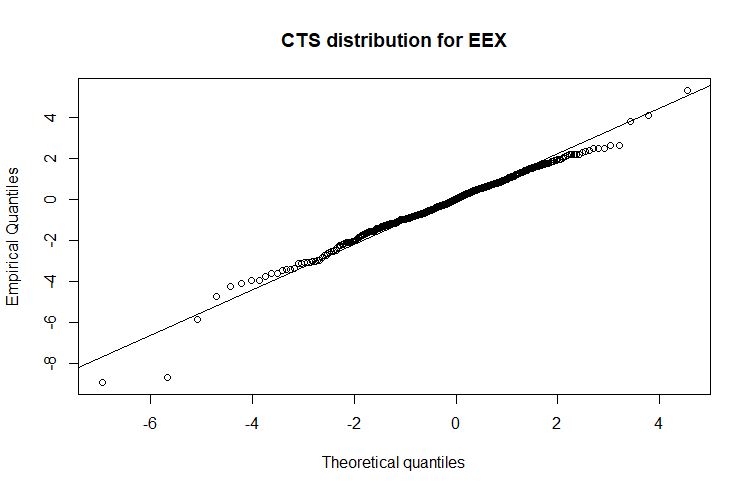}}\par

	\subfloat[][]{\includegraphics[width=.65\textwidth]{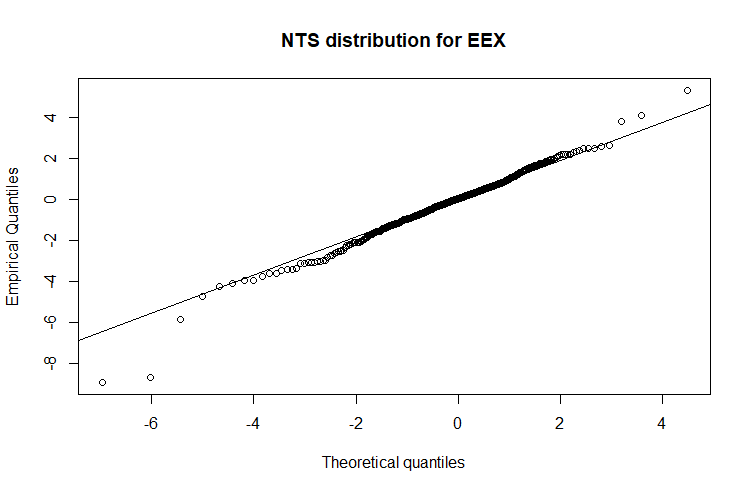}}\par
  
  \caption{QQ-Plots comparing the goodness-of-fit of the stable, the CTS and the NTS distributions for the EEX. The solid line is the reference line.}
  \label{fig:EEXqq}
\end{figure*}

\section{Conclusions and Future Work}\label{sec:conclusion}
This paper derived asymptotic efficiency results for parametric estimation methods of the tempered stable subordinator, the classical tempered stable distribution, and the normal tempered stable distribution. We conducted a Monte Carlo study to establish finite sample properties. It turned out that the generalized methods of moments estimator with a continuum of moment conditions and the maximum likelihood estimator outperformed the other methods. We discussed why tempered stable distributions are relevant in financial applications.
Asymptotic results for other tempering functions or the derivation of a set of conditions to ensure asymptotic efficiency for general tempering functions are a subject of future work. 

\section*{Acknowledgements}
Financial support of the German Research Foundation (Deutsche Forschungsgemeinschaft, DFG) via the project 455257011 is gratefully acknowledged.

The author is grateful to Christoph Hanck for valuable comments which helped to substantially improve this paper. Full responsibility is taken for
all remaining errors.

\bibliography{bibliography}
\bibliographystyle{agsm}

\appendix

\section{Appendix: Proofs}\label{app:proofs}

\begin{proof}[Proof of Theorem \ref{thm:ML}]
\begin{enumerate}
\item[(a)] We now prove the asymptotic normality of the TSS by verifying the conditions (i)--(vi) of Assumption \ref{ass:ML}. In order to do so, we follow \cite{DuMouchel1973} who proved asymptotic normality for stable distributions. He showed the statement for stable distributions which are not totally skewed. More precisely, he explicitly excluded the subordinator case in his paper. However, we do not contradict his statements here. This is because (in our notation) the skewness parameter $\rho$ is fixed at $\rho=1$ (i.e., total positive skewness) and does not need to be estimated in our setting. (In his notation, the skewness parameter $\rho$ fulfills $|\rho|<\min\{\alpha,2-\alpha\}$, i.e., total positive skewness means that $\rho$ is fixed at $\rho=\alpha$.) In the setting of \cite{DuMouchel1973}, the skewness parameter $\rho$ still needed to be estimated, for which the usual asymptotics do not work. In our setting, the proof for stable subordinators follows analogously by checking Conditions 1-6 of \cite{DuMouchel1973} with minor changes summarized now. We only estimate the stability index and the scale parameter here so the Fisher information is a $2\times2$-matrix. Conditions 1,2,4-6 of \cite{DuMouchel1973} follow analogously. For his Condition 3, we refer to (vi) below where we prove the statement for the TSS and for the stable subordinator as an intermediate step.

We now check the conditions of Assumption \ref{ass:ML}. \citet[][Section 3.2.1]{Grabchak2016} proved condition (i). (ii) $\Theta$ is compact by construction. The density function \eqref{eq:dTSS} for the TSS distribution is twice continuously differentiable with respect to the parameters by the twice continuous differentiability of the density of the stable subordinator \cite[][Condition 1]{DuMouchel1973} and the continuous differentiability of the gamma function for $\alpha>0$. Also $f_{TSS}(y;\theta)>0$ for all $y>0$. This implies (iii). 

To show (iv), we make use of \eqref{eq:dTSS} and bound the partial derivatives
\begin{equation}\label{eq:boundTSSderiv}
\left|\frac{\partial f_{TSS}(y;\theta)}{\partial\vartheta}\right|\le C_{\theta}\cdot\begin{cases} f_{TSS}(y;\theta)+\big|\frac{\partial f_{S(\alpha,\delta)}(y)}{\partial\alpha}\big|,& \vartheta=\alpha,\\ f_{TSS}(y;\theta)+\big|\frac{\partial f_{S(\alpha,\delta)}(y)}{\partial\delta}\big|,&\vartheta=\delta,\\ f_{TSS}(y;\theta) +yf_{TSS}(y;\theta), &\vartheta=\lambda,
\end{cases}
\end{equation}
where $C_{\theta}$ is a positive, finite constant which depends on $\theta$ but not on $y$. Because $\theta\in\Theta$ which is compact and because $C_{\theta}$ can be taken to be continuous in $\theta$ which follows from (iii), $C_{\theta}$ can be bounded by a constant $C$. The result now follows because the expected value of a TSS distribution exists and because the derivatives of stable densities are integrable as discussed in the proof of Condition 4 in \cite{DuMouchel1973}. Integrability also holds when taking the supremum as can be verified analogously to Condition 3 in \cite{DuMouchel1973} by analyzing the supremum of the series representation (5.1). The second part of (iv) follows analogously by using
\begin{equation}\label{eq:boundTSSderiv2}
\left|\frac{\partial^2 f_{TSS}(y;\theta)}{\partial\vartheta^2}\right|\le C_{\theta}\cdot\begin{cases} f_{TSS}(y;\theta)+\big|\frac{\partial f_{S(\alpha,\delta)}(y)}{\partial\alpha}\big|+\big|\frac{\partial^2 f_{S(\alpha,\delta)}(y)}{\partial\alpha^2}\big|,& \vartheta=\alpha,\\ f_{TSS}(y;\theta)+\big|\frac{\partial f_{S(\alpha,\delta)}(y)}{\partial\delta}\big|+\big|\frac{\partial^2 f_{S(\alpha,\delta)}(y)}{\partial\delta^2}\big|,&\vartheta=\delta,\\ f_{TSS}(y;\theta) +yf_{TSS}(y;\theta)+y^2f_{TSS}(y;\theta), &\vartheta=\lambda
\end{cases}
\end{equation}
and that the second moment of the TSS distribution is finite.

In order to show (v), we follow \citet[][Condition 6']{DuMouchel1973} and show the equivalent condition that for every $\theta\in\Theta$ and for every $a=(a_1,a_2,a_3)\in\mathbb{R}^3$ the function
\begin{equation}
g(a,y)=a_1\frac{\partial f_{TSS}(y;\theta)}{\partial\alpha}+a_2\frac{\partial f_{TSS}(y;\theta)}{\partial\delta}+a_3\frac{\partial f_{TSS}(y;\theta)}{\partial\lambda}
\end{equation}
is identically 0 for all $y$ if and only if $a_1=a_2=a_3=0$. It holds 
\begin{equation}\label{eq:intdiffchar}
g(a,y)=\int_{\mathbb{R}}\mathrm{e}^{-\mathrm{i}yt}\left(a_1\phi_1+a_2\phi_2+a_3\phi_3\right)\upd t=\int_{\mathbb{R}}\mathrm{e}^{-\mathrm{i}yt}\phi(a,t)\upd t,
\end{equation} 
where $\phi$ is a linear combination of derivatives of the TSS characteristic function given in \eqref{eq:charTSS}, i.e., the partials are
\begin{align}
\phi_1=\frac{\partial\varphi_{TSS}(t;\theta)}{\partial\alpha}&=\varphi_{TSS}(t;\theta)\delta\Gamma(-\alpha)\\
&\qquad\cdot\left(\left((\lambda-\mathrm{i}t)^{\alpha}\log(\lambda-\mathrm{i}t)-\lambda^{\alpha}\log(\lambda)\right)-\left((\lambda-\mathrm{i}t)^{\alpha}-\lambda^{\alpha}\right)\psi(-\alpha)\right),\\
\phi_2=\frac{\partial\varphi_{TSS}(t;\theta)}{\partial\delta}&=\varphi_{TSS}(t;\theta)\Gamma(-\alpha)\left((\lambda-\mathrm{i}t)^{\alpha}-\lambda^{\alpha}\right),\\
\phi_3=\frac{\partial\varphi_{TSS}(t;\theta)}{\partial\lambda}&=\varphi_{TSS}(t;\theta)\delta\Gamma(-\alpha)\left(\alpha(\lambda-\mathrm{i}t)^{\alpha-1}-\alpha\lambda^{\alpha-1}\right),
\end{align}
where $\psi(x)=\frac{\Gamma^{\prime}(x)}{\Gamma(x)}$ denotes the digamma function. To show that the interchange of integration and differentiation in \eqref{eq:intdiffchar} is permitted we can use inequality (A3) of Lemma A.7 in \cite{Xia2022}, i.e., it exists a constant $C>0$ such that $|\varphi_{TSS}(t;\theta)|\le \mathrm{e}^{-C(|t|^{\alpha}\wedge|t|^2)}$ for the TSS distribution. This implies that both $\int_{\mathbb{R}}|\mathrm{e}^{-\mathrm{i}yt}|\ |\varphi_{TSS}(t;\theta)|\upd t<\infty$ and $\int_{\mathbb{R}}|\mathrm{e}^{-\mathrm{i}yt}|\ |\phi(a,t)|\upd t<\infty$. Then $g(a,y)\equiv0$ iff $\phi(a,t)\equiv0$ because Fourier transforms uniquely determine functions. The characteristic function given in \eqref{eq:charTSS} is non-zero for each $t$ which implies that it is sufficient to study the latter terms of $\phi_1,\phi_2,\phi_3$. $\phi_1$ is a linear combination of $\left((\lambda-\mathrm{i}t)^{\alpha}\log(\lambda-\mathrm{i}t)-\lambda^{\alpha}\log(\lambda)\right)$ and $\left((\lambda-\mathrm{i}t)^{\alpha}-\lambda^{\alpha}\right)$, $\phi_2$ is a multiple of $\left((\lambda-\mathrm{i}t)^{\alpha}-\lambda^{\alpha}\right)$, and $\phi_3$ is a multiple of $\left(\alpha(\lambda-\mathrm{i}t)^{\alpha-1}-\alpha\lambda^{\alpha-1}\right)$. These terms are linearly independent which can be seen, e.g., by checking that the Wronskian determinant is non-zero.
Since $\phi_1$ is the only part of the linear combination with term $(\lambda-\mathrm{i}t)^{\alpha}\log(\lambda-\mathrm{i}t)$ it follows $a_1=0$, otherwise, $\phi$ would not be equal to 0 for each $t$. Because $\phi_2$ is the only remaining part of the linear combination with term $\left((\lambda-\mathrm{i}t)^{\alpha}-\lambda^{\alpha}\right)$ and $\phi_3$ is the only part of the linear combination with term $\left(\alpha(\lambda-\mathrm{i}t)^{\alpha-1}-\alpha\lambda^{\alpha-1}\right)$, it is necessary for $\phi(a,t)=0$ for all $t$ that also $a_2=a_3=0$. 

To show (vi), we recall that \citet[][Condition 3]{DuMouchel1973} proved the statement for the stable distribution. However, he explicitly excluded totally skewed stable distributions if both the stability index and the skewness parameter need to be estimated. We adapt his proof and show the statement for fixed and known total positive skewness, i.e., the stable subordinator.
By \eqref{eq:dTSS},
\begin{equation}
\label{eq:logdTSS}
\frac{\partial^2}{\partial\theta\partial\theta^{\prime}}\log f_{TSS}(y;\theta)=\frac{\partial^2}{\partial\theta\partial\theta^{\prime}}\left(-\lambda y-\lambda^{\alpha}\delta\Gamma(-\alpha)\right)+\frac{\partial^2}{\partial\theta\partial\theta^{\prime}}\log f_{S(\alpha,\delta)}(y),
\end{equation}
the statement for the TSS distribution then follows immediately.
Following \citet[][Proof of Condition 3]{DuMouchel1973}, the absolute value of the matrix-elements of $\frac{\partial^2\log f_{S(\alpha,\delta)}(y)}{\partial\theta^2}$ have a maximum $C(y)$, for fixed $y$ and $\theta$ in the compact set $\Theta$. Therefore, it is only necessary to study the behavior of $C(y)$ at its limit points. \cite{DuMouchel1973} showed that for $y\to\infty$ the corresponding $C(y)$ for the stable distribution is of the order $O(\log^2|y|)$ by deriving the series representation given in (5.1) of \cite{DuMouchel1973} and computing the order of the bound. The result for $y\to\infty$ transfers to our situation.
Since the totally skewed subordinator has no left tail we have to study the behavior of $C(y)$ for $y\to0$. However, the series representation given in \eqref{eq:dTSSseries} does not work for $y\to0$ (with or without the tempering term). There is another series representation which holds for $1<\alpha<2$ given by
\begin{equation}
\label{eq:dTSSseries2}
f_{S(\alpha,\delta)}(y)=\frac{-1}{\pi}\sum_{k=1}^{\infty}\frac{(-1)^k}{k!}\Gamma(1+k/\alpha)\Gamma(1-\alpha)^{-k/\alpha}\left(\frac{\delta}{\alpha}\right)^{-k/\alpha}y^{k-1}\sin(\pi k),
\end{equation}
see \cite{Bergstrom1952,Nolan2020}, adapted to our parametrization, for $y\to0$. For $\alpha<1$, the series is divergent but the partial sum of the first $n$ terms is an asymptotic expansion for all $n$ such that the remainder is of order $O(y^n)$ for $y\to0$. For $1<\alpha<2$, the series \eqref{eq:dTSSseries2} is absolutely convergent which allows changing sum and differentiation with respect to $y$ and the parameters. The coefficients for the series representation for the derivatives are cumbersome and thus omitted. We now use the same argument as \cite[][Section 5]{DuMouchel1973} that coefficients for asymptotic expansions are unique which implies that the derivatives have the same coefficients for $\alpha<1$ as for $\alpha>1$. The derivatives with respect to $\theta$ are lengthy but not complicated, in particular, they do not interfere with $y$. Therefore, $C(y)=O(1)$ for $y\to0$. All in all, this implies that $\int_0^{\infty} C(y)f_{S(\alpha,\delta)}(y)\upd y<\infty$ and hence $\int_0^{\infty} C(y)f_{TSS}(y;\theta)\upd y<\infty$ for each $\theta\in\Theta$.

%For $y\to0$, we use the approximation $-\frac{1-\alpha}{\alpha}\delta\Gamma(1-\alpha)^{\frac{1}{1-\alpha}}y^{-\frac{\alpha}{1-\alpha}}$ of $\log f_{TSS}(y;\theta)$ by \cite{Kuchler2013}, and observe that $C(y)$ is of the order $\log^2|y|$.

%\item[(b)] The proof is analogously to (a) but replacing relation \eqref{eq:dTSS} with \eqref{eq:denscomposition2}.

\item[(b)] Conditions (i) and (ii) work just as in (a). For (iii), the differentiability follows by \eqref{eq:convolutionCTS2} and \eqref{eq:convolutionCTS3} below and Theorem 28.4 in \cite{Sato1999levy} that gives differentiability in $x$ which is needed to guarantee differentiability in $\mu$. Moreover, Theorems 24.10 and 53.1 in \cite{Sato1999levy} imply the density $f_{CTS}(x;\theta)$ is unimodal and strictly positive. %we follow \citet[][Theorem 7.8]{Kuchler2013} who showed that the density $f_{CTS}(x;\theta)$ is smooth and moreover unimodal and strictly positive for $\alpha\in(0,1)$. The proof for $\alpha\in [1,2)$ (the infinite variation case) is the same but replacing class $\mathrm{III}_6$ by class $\mathrm{III}_7$ in the sense of \citet[][Theorem 1.3 (ix)]{Sato1978}.

For (iv), we use that by \eqref{eq:convolutionCTS}
\begin{equation}
\label{eq:convolutionCTS2}
f_{CTS}(x+\mu;\theta)=\int_{-\infty}^{\infty}f_+(x+y)f_-(y)\upd y,
\end{equation}
where $f_+$ is the density of $Y_+\sim TS^{\prime}(\alpha,\delta_+,\lambda_+)$ and $f_-$ is the density of $Y_-\sim TS^{\prime}(\alpha,\delta_-,\lambda_-)$. We remark that for $\alpha\in(0,1)$ the lower limit of the integral is technically not $-\infty$ but the minimum value on which both $f_+$ and $f_-$ are non-zero. This value depends on $\theta$ and is the reason why MLE asymptotics do not hold for the TS' distribution if $\alpha\in(0,1)$. Then, the functions $f_+(y)$ and $f_-(y)$ are identical zero for all $y$ smaller than $\Gamma(1-\alpha)\delta_+\lambda_+^{\alpha-1}$ and $\Gamma(1-\alpha)\delta_-\lambda_-^{\alpha-1}$, respectively. Thus the right tail has a different behavior than the left if $\alpha\in(0,1)$. If $\alpha\in[1,2)$, $f_+$ and $f_-$ are strictly positive on $\mathbb{R}$ and we can use the same bounds for both left and right tails. We will make repeated use of this case distinction throughout part (b) of the proof. Although Assumption \ref{ass:ML}.(iii) does not hold for the TS' density if $\alpha\in(0,1)$, it is possible to derive Assumption \ref{ass:ML}.(iv)  for $\alpha\in(0,2)$ analogously to part (a) for the TSS density by replacing relation \eqref{eq:dTSS} with \eqref{eq:denscomposition2}. This and the fact that $\sup_{\theta\in\Theta}\Big|\Big|\frac{\partial f_{\pm}(x)}{\partial \theta}\Big|\Big|$ and $\sup_{\theta\in\Theta}\Big|\Big|\frac{\partial^2 f_{\pm}(x)}{\partial \theta\partial\theta^{\prime}}\Big|\Big|$ can be bounded in $x$ ensures that we can change the order of differentiation and integration in \eqref{eq:convolutionCTS2} (the product of two integrable functions is integrable if at least one function is bounded). Thus,
\begin{align}\label{eq:convolutionCTS3}
\frac{\partial f_{CTS}(x;\theta)}{\partial\theta}&=\frac{\partial}{\partial\theta}\int_{-\infty}^{\infty}f_+(x+y-\mu)f_-(y)\upd y\\
&=\int_{-\infty}^{\infty}\frac{\partial}{\partial\theta}\left(f_+(x+y-\mu)f_-(y)\right)\upd y\\
&=\int_{-\infty}^{\infty}\left(\frac{\partial f_+(x+y-\mu)}{\partial\theta}f_-(y)+\frac{\partial f_-(y)}{\partial\theta}f_+(x+y-\mu)\right)\upd y.
\end{align}
In the second line we used that we can change the order of integration and differentiation. Hence,
\begin{align}
&\ \int_{-\infty}^{\infty}\sup_{\theta\in\Theta}\bigg|\bigg|\frac{\partial f_{CTS}(x;\theta)}{\partial \theta}\bigg|\bigg|\upd x\\
\le&\ \int_{-\infty}^{\infty}\int_{-\infty}^{\infty}\sup_{\theta\in\Theta}\bigg|\bigg|\frac{\partial f_+(x+y-\mu)}{\partial\theta}f_-(y)\bigg|\bigg|\upd y\upd x
+ \int_{-\infty}^{\infty}\int_{-\infty}^{\infty}\sup_{\theta\in\Theta}\bigg|\bigg|\frac{\partial f_-(y)}{\partial\theta}f_+(x+y-\mu)\bigg|\bigg|\upd y\upd x\\
=&\ \int_{-\infty}^{\infty}\int_{-\infty}^{\infty}\sup_{\theta\in\Theta}\bigg|\bigg|\frac{\partial f_+(x+y-\mu)}{\partial\theta}f_-(y)\bigg|\bigg|\upd x\upd y
+ \int_{-\infty}^{\infty}\int_{-\infty}^{\infty}\sup_{\theta\in\Theta}\bigg|\bigg|\frac{\partial f_-(y)}{\partial\theta}f_+(x+y-\mu)\bigg|\bigg|\upd x\upd y\\
\le & \ \int_{-\infty}^{\infty}\sup_{\theta\in\Theta}|f_-(y)|\int_{-\infty}^{\infty}\sup_{\theta\in\Theta}\bigg|\bigg|\frac{\partial f_+(x+y-\mu)}{\partial\theta}\bigg|\bigg|\upd x\upd y \\
&\qquad+ \int_{-\infty}^{\infty}\sup_{\theta\in\Theta}\bigg|\bigg|\frac{\partial f_-(y)}{\partial\theta}\bigg|\bigg|\int_{-\infty}^{\infty}\sup_{\theta\in\Theta}\left|f_+(x+y-\mu)\right|\upd x\upd y\label{eq:innerintegral}\\
<& \ \infty,
\end{align}
where we used Fubini's Theorem in the second step. To see the finiteness \eqref{eq:innerintegral} we will in the following find bounds for the two series representations/asymptotic expansions \eqref{eq:dTSSseries} and \eqref{eq:dTSSseries2} that are independent of $\theta$ and show that they are integrable. We need to discuss the several integrals separately. 

For the inner integral of the left term in \eqref{eq:innerintegral}, we use the asymptotic expansions for the integrand as in part (a) and the relation \eqref{eq:denscomposition2}. 
\begin{itemize}
\item Its right tail (without the sup) can be bounded by $C_{\theta}\log(x+y)(x+y)^{-\alpha-1}$, where $C_{\theta}$ is a constant only depending on $\theta$. This follows by equation (5.1) of \cite{DuMouchel1973} and by differentiating \eqref{eq:denscomposition2}. (Note that the exponential function of the tempering part in \eqref{eq:denscomposition2} can be bounded by $C_{\theta}$.) The bound $C_{\theta}\log(x+y)(x+y)^{-\alpha-1}$ attains its supremum for $\alpha=\varepsilon$ in the right tail (recalling that $\alpha\in[\varepsilon,2-\varepsilon]$).
\item For $\alpha\in(0,1)$ the left tail is bounded by $C_{\theta}(x+y)$ which follows by bounding the derivative of \eqref{eq:dTSSseries2} (which exact form we omit, see the discussion below equation \eqref{eq:dTSSseries2}). Thus for $\alpha\in(0,1)$, the inner integral exists and is bounded by $C(y^2+\log(y)/y^{\varepsilon})$, $x$-integral of the sum of both bounds, where $C$ is a constant. For $\alpha\in[1,2)$, the left tail has the same bound as the right tail (because the (5.1) of \cite{DuMouchel1973} holds for both tails) and hence the inner integral is bounded by $C\log(y)/y^{\varepsilon}$ (by integrating over $x$). 
\end{itemize}
In both cases, the bound is integrable with respect to the TS' density and hence the left double integral is finite. Next, we turn to the right inner integral.
\begin{itemize}
\item We see that the function is bounded on the right and left (if $\alpha\in[1,2)$) tail by $C_{\theta}(x+y)^{-\alpha-1}$ via the expansion \eqref{eq:dTSSseries} and \eqref{eq:denscomposition2} and for the left tail (if $\alpha\in(0,1)$) by a constant $C_{\theta}$ via \eqref{eq:dTSSseries2} and \eqref{eq:denscomposition2}. Thus, the inner integral is bounded by $C(y+1/y^{\varepsilon})$ if $\alpha\in(0,1)$ and by $C/y^{\varepsilon}$ if $\alpha\in[1,2)$.
\item For the right outer integral, we again have to consider the parameter derivatives of \eqref{eq:dTSSseries} (right tail and left tail if $\alpha\in[1,2)$) and of \eqref{eq:dTSSseries2} (left tail if $\alpha\in(0,1)$). All in all, this can be bounded by $C(y^{-2\varepsilon-1}\log(y))$ (right tail and left tail if $\alpha\in[1,2)$) and by $C$ (left tail if $\alpha\in(0,1)$).
\end{itemize}
Therefore, both double integrals are finite. The second part of (iv) follows analogously by first interchanging the order of integration and second derivative and then analyzing the behavior of the four resulting double integrals using our asymptotic expansions.

(v) follows analogously to part (a), in particular, if we denote $(\phi_1,\ldots,\phi_6)^{\prime}=\frac{\partial\varphi_{CTS}(t;\theta)}{\partial\theta}$ we observe that $\phi_1$ is the only linear combination with terms $ (\lambda_+-\mathrm{i}t)^{\alpha}\log(\lambda_+-\mathrm{i}t)$ and $ (\lambda_-+\mathrm{i}t)^{\alpha}\log(\lambda_-+\mathrm{i}t)$, which implies $a_1=0$. $\phi_2$ is a linear combination of $t$ and $(\lambda_+-\mathrm{i}t)^{\alpha}$, $\phi_3$ is a linear combination of $t$ and $(\lambda_-+\mathrm{i}t)^{\alpha}$, $\phi_4$ is a linear combination of $t$ and $(\lambda_+-\mathrm{i}t)^{\alpha-1}$, $\phi_5$ is a linear combination of $t$ and $(\lambda_-+\mathrm{i}t)^{\alpha-1}$, and $\phi_6$ is a multiple of $t$. This implies $a_2=a_3=a_4=a_5=0$ and hence $a_6=0$.

For (vi), we have to study the limiting behavior of $C(x)$ which is the function which bounds each of the elements of $\frac{\partial^2}{\partial\theta\partial\theta^{\prime}}\log f_{CTS}(x;\theta)$ in absolute value. We make use of the bound 
\begin{equation}
\label{eq:secondlogderiv}
\frac{\partial^2}{\partial\theta\partial\theta^{\prime}}\log f_{CTS}(x;\theta)\le\frac{\frac{\partial^2}{\partial\theta\partial\theta^{\prime}}f_{CTS}(x;\theta)}{f_{CTS}(x;\theta)}.
\end{equation}
We recall the convolution property \eqref{eq:convolutionCTS2}, and that we are allowed to interchange integration and differentiation we therefore need to study 
\begin{equation}
\label{eq:beforeproductrule}
\frac{\partial^2}{\partial\theta\partial\theta^{\prime}}f_+(x+y-\mu)f_-(y)
\end{equation}
for that we can apply the product rule for differentiation. We have to use the same trick as above to make a case distinction for $\alpha<1$ and for $\alpha\ge1$ to analyze the behavior of the right and of the left tail separately. The right tail and the left tail of \eqref{eq:beforeproductrule} for $\alpha\ge1$ can be analyzed with the asymptotic expansion \eqref{eq:dTSSseries} analogously as above and its elements can be bounded by $C_{\theta}(x+y)^{-1-\alpha}\log(x+y)^2$. The left tail for $\alpha\in(0,1)$ can be bounded by $C_{\theta}$. Therefore, the $\mathrm{d}y$-integral over \eqref{eq:beforeproductrule} can be bounded by $C_{\theta}x^{-\alpha}\log(x)$. Given that series \eqref{eq:dTSSseries} is alternating, we can also find a lower bound for the denominator in \eqref{eq:secondlogderiv} such that \eqref{eq:secondlogderiv} can be bounded by $C_{\theta}\log(x)^2$. Since $\Theta$ is compact, the expectation in (vi) is finite which completes the proof for (b).

\item[(c)] (i) and (ii) as above. For the remainder we use the subordination property of the NTS distribution, i.e.,
\begin{equation}
\label{eq:dNTS}
f_{NTS}(z;\theta)=\int_0^{\infty}f_N(z;\mu+\beta y,y)f_{TSS}(y;\alpha,\delta,\lambda)\upd y,
\end{equation}
where $f_N(z;m,s^2)$ denotes the density of the normal distribution with mean $m$ and variance $s^2$.
Similarly to \cite{Massing2019a}, we can show that $\left|\frac{\partial f_N(z;\mu+\beta y,y)}{\partial \mu}\right|\le\frac{C}{y}$ and $\left|\frac{\partial f_N(z;\mu+\beta y,y)}{\partial \beta}\right|\le C$ for all $\theta\in\Theta$. Additionally, because $|f_N(z;\mu+\beta y,y)|\le\frac{C}{\sqrt{y}}$, we have for $\vartheta\in(\alpha,\delta,\lambda)$ that $\left|\frac{\partial f_N(z;\mu+\beta y,y)f_{TSS}(y;\alpha,\delta,\lambda)}{\partial \vartheta}\right|\le\frac{C}{\sqrt{y}}\left|\frac{\partial f_{TSS}(y;\alpha,\delta,\lambda)}{\partial \vartheta}\right|$.
Thus,
\begin{equation}
\frac{\partial f_{NTS}(z;\theta)}{\partial\theta}=\int_0^{\infty}\frac{\partial}{\partial\theta}(f_N(z;\mu+\beta y,y)f_{TSS}(y;\alpha,\delta,\lambda))\upd y
\end{equation}
holds by property (iv) for the TSS distribution. This implies (iii).

(iv) follows analogously as in (b) because
\begin{align}
&\ \int_{-\infty}^{\infty}\sup_{\theta\in\Theta}\bigg|\bigg|\frac{\partial f_{NTS}(z;\theta)}{\partial \theta}\bigg|\bigg|\upd z\\
\le&\ \int_{-\infty}^{\infty}\int_0^{\infty}\sup_{\theta\in\Theta}\bigg|\bigg|\frac{\partial}{\partial\theta}\left(f_N(z;\mu+\beta y,y)f_{TSS}(y;\alpha,\delta,\lambda)\right)\bigg|\bigg|\upd y\upd z\\
=&\ \int_0^{\infty}\int_{-\infty}^{\infty}\sup_{\theta\in\Theta}\bigg|\bigg|\frac{\partial}{\partial\theta}\left(f_N(z;\mu+\beta y,y)f_{TSS}(y;\alpha,\delta,\lambda)\right)\bigg|\bigg|\upd z\upd y.\label{eq:boundNTSproof}
\end{align}
In order to bound \eqref{eq:boundNTSproof}, we discuss the partial derivatives separately. For $\frac{\partial}{\partial\vartheta}$, $\vartheta\in(\alpha,\delta,\lambda)$,
\begin{align}
&\ \int_0^{\infty}\int_{-\infty}^{\infty}\sup_{\theta\in\Theta}\bigg|\frac{\partial}{\partial\vartheta}\left(f_N(z;\mu+\beta y,y)f_{TSS}(y;\alpha,\delta,\lambda)\right)\bigg|\upd z\upd y\\
\le&\ \int_0^{\infty}\sup_{\theta\in\Theta}\bigg|\frac{\partial}{\partial\vartheta}f_{TSS}(y;\alpha,\delta,\lambda)\bigg|\int_{-\infty}^{\infty}\sup_{\theta\in\Theta}|f_N(z;\mu+\beta y,y)|\upd z\upd y\\
\le& \ \int_0^{\infty}\frac{C}{\sqrt{y}}\sup_{\theta\in\Theta}\bigg|\frac{\partial}{\partial\vartheta}f_{TSS}(y;\alpha,\delta,\lambda)\bigg|\upd y\\
<&\ \infty.
\end{align}
The third line follows because $\Theta$ is a compact set such that the integrands can be bounded away from zero and $\infty$. We can decompose the inner integral
$\int_{-\infty}^{\infty}\sup_{\theta\in\Theta}|f_N(z;\mu+\beta y,y)|\upd z=\int_{z_*}^{z^*}\sup_{\theta\in\Theta}|f_N(z;\mu+\beta y,y)|\upd z+\int_{\mathbb{R}\backslash[z_*,z^*]}\sup_{\theta\in\Theta}|f_N(z;\mu+\beta y,y)|\upd z$.
The first integral attains the maximum value $\frac{C}{\sqrt{y}}$. The second integral is bounded by one because we integrate the normal density with the same value $\theta^*\in\Theta$ which makes the exponent in the exponential function maximal for all $z\in\mathbb{R}\backslash[z_*,z^*]$. The finiteness in the fourth line thus follows analogously to part (a) condition (iv) (using that we can find a bound independent of the parameters for the series representations that is integrable). For $\frac{\partial}{\partial\mu}$, \eqref{eq:boundNTSproof} is bounded by
\begin{align}
&\ \int_0^{\infty}\sup_{\theta\in\Theta}|f_{TSS}(y;\alpha,\delta,\lambda)|\int_{-\infty}^{\infty}\sup_{\theta\in\Theta}\bigg|\frac{\partial}{\partial\mu}f_N(z;\mu+\beta y,y)\bigg|\upd z\upd y\\
\le& \ C\int_0^{\infty}\frac{1}{\sqrt{y}}\sup_{\theta\in\Theta}|f_{TSS}(y;\alpha,\delta,\lambda)|\upd y\\
<&\ \infty.
\end{align}
The second line follows because the inner integral is easy to analyze (derive w.r.t.~$\mu$, take the sup on a compact set and integrate w.r.t.~$z$) and bounded by $\frac{C}{\sqrt{y}}$. Thus, we obtain that the double integral is finite which follows by the series representation \eqref{eq:dTSSseries} for the right tail and the asymptotic expansion \eqref{eq:dTSSseries2} near the origin. 
For $\frac{\partial}{\partial\beta}$,
\begin{align}
&\ \int_0^{\infty}\sup_{\theta\in\Theta}|f_{TSS}(y;\alpha,\delta,\lambda)|\int_{-\infty}^{\infty}\sup_{\theta\in\Theta}\bigg|\frac{\partial}{\partial\beta}f_N(z;\mu+\beta y,y)\bigg|\upd z\upd y\\
\le& \ C\int_0^{\infty}\sqrt{y}\sup_{\theta\in\Theta}|f_{TSS}(y;\alpha,\delta,\lambda)|\upd y\\
<&\ \infty.
\end{align}
The second part of (iv) follows analogously. (v) follows in a similar fashion as in (a) and (b).

As in (a)\&(b), it is for (vi) enough to consider the behavior of $C(z)$ for $z\to\pm\infty$, where $C(z)$ is the function which bounds the elements of $\frac{\partial^2}{\partial\theta\partial\theta^{\prime}}\log f_{NTS}(z;\theta)$ in absolute value. With the same arguments as above it holds that
\begin{equation}\label{eq:cvi}
\frac{\partial^2 f_{NTS}(z;\theta)}{\partial\theta\partial\theta'}=\int_0^{\infty}\frac{\partial^2}{\partial\theta\partial\theta'}\left(f_N(z;\mu+\beta y,y)\mathrm{e}^{-\lambda y-\lambda^{\alpha}\delta\Gamma(-\alpha)}f_{S}(y;\alpha,\delta)\right)\upd y.
\end{equation}
We use that by the series representation for the parameter derivatives of \cite{DuMouchel1973} the second derivatives of the stable distribution are bounded by $C_{\theta}y^{-\alpha-1}\log(y)^2$ for large $y$. Using this bound in \eqref{eq:cvi} and integrating w.r.t.~$y$ yields that $C(z)=O(z^{-1-2\varepsilon}\log(|z|)^2)$, where $\varepsilon$ is the lower boundary of admissible values for $\alpha\in[\varepsilon,1-\varepsilon]$ (i.e., the value of $\alpha$ that maximizes the bound). This implies (vi).
\end{enumerate}
\end{proof}

\begin{proof}[Proof of Theorem \ref{thm:CGMM}]
We check the conditions of Assumption \ref{ass:CGMM} for the tempered stable subordinator. The other two cases follow analogously. (i) $\Theta$ is compact by construction and $\pi$ is chosen in such a way that it fulfills (ii). (iii) holds if 
\begin{equation}\label{eq:proof:thm:CGMM}
\exp\left(\delta\Gamma(-\alpha)\left((\lambda-\mathrm{i}t)^{\alpha}-\lambda^{\alpha}\right)\right)=\exp\left(\delta_0\Gamma(-\alpha_0)\left((\lambda_0-\mathrm{i}t)^{\alpha_0}-\lambda_0^{\alpha_0}\right)\right)
\end{equation}
holds for each $t$. Recall that the \Levy triplet by the L\'evy-Khtinchine representation uniquely determines the distribution. Thus, \eqref{eq:proof:thm:CGMM} is equivalent to 
\begin{equation}\label{eq:proof:thm:CGMM2}
\frac{\mathrm{e}^{-\lambda r}\delta}{r^{1+\alpha}}\mathds{1}_{(0,\infty)}(r)=\frac{\mathrm{e}^{-\lambda_0 r}\delta_0}{r^{1+\alpha_0}}\mathds{1}_{(0,\infty)}(r)
\end{equation}
for each $r$. By taking the logarithm, \eqref{eq:proof:thm:CGMM2} is equivalent to the linear independence of $1$, $\log r$ and $r$, which is true.
%Assuming that $\alpha$ is not unique, i.e., let $\alpha_1\neq\alpha_0$, immediately implies that the characteristic function cannot be identical for all $t$. This uniquely determines $\alpha=\alpha_0$. This implies that $\lambda=\lambda_0$ because $\lambda$ is the only parameter which is responsible for the shift in the exponent of the characteristic function. Therefore also $\delta=\delta_0$.
(iv)--(vi) have been shown in the proof of Theorem \ref{thm:ML}.
\end{proof}

\begin{proof}[Proof of Theorem \ref{thm:GMC}]
\begin{enumerate}
\item[(a)]
The crucial assumption of \citet[][Theorem 2.6]{Newey1994} for consistency of a GMM estimator is that $W\mathbb{E}_{\theta_0}[g(X_j;\theta)]=0$ only if $\theta=\theta_0$, where $W=\Omega^{-1}$ if it is invertible. \citet[][Lemma 6.1]{Kuchler2013} showed that cumulant matching for the CTS distribution holds locally by the local inverse function theorem. In their Lemma 5.4, they explicitly solved the cumulant matching for the parameters for the TSS distribution. We here follow the standard GMM theory by verifying that $G=\left(g_{mi}\right)_{1\le m\le p,1\le i\le3}=\left.\mathbb{E}_{\theta_0}\left[\frac{\partial g(X;\theta)}{\partial\theta}\right]\right|_{\theta=\theta_0}$ is of full column rank (in this case 3 which is the number of parameters). This implies local identification, see \citet[][Theorem 2]{Rothenberg1971}. Recall that function $g$ is given in \eqref{eq:functiong}, that first moments are equal to the first cumulants, and that the cumulants for the TSS are given in \eqref{eq:cumsTSS}. To show that $G$ has full rank, it is enough to consider the first three row of $G$. This is because if the matrix with the first three row has full rank then this is the maximal number of linearly independent rows for the restricted matrix as well as of G (because the column rank can be at most 3). The entries of the first three rows are
\begin{align}\label{eq:firstrowG}
g_{11}&=\delta  \lambda ^{\alpha -1} \Gamma (1-\alpha ) \log (\lambda )-\delta  \lambda ^{\alpha -1} \Gamma (1-\alpha ) \psi (1-\alpha )\\
g_{12}&=\lambda ^{\alpha -1} \Gamma (1-\alpha )\\
g_{13}&= (\alpha -1) \delta  \lambda ^{\alpha -2} \Gamma (1-\alpha ) \\
g_{21}&= 2 \delta ^2 \lambda ^{2 \alpha -2} \Gamma (1-\alpha )^2 \log (\lambda )-2 \delta ^2 \lambda ^{2 \alpha -2} \Gamma (1-\alpha )^2 \psi (1-\alpha )+\delta  \lambda ^{\alpha -2}
   \Gamma (2-\alpha ) \log (\lambda )\\
	&\quad-\delta  \lambda ^{\alpha -2} \Gamma (2-\alpha ) \psi (2-\alpha )\\
g_{22}&= 2 \delta  \lambda ^{2 \alpha -2} \Gamma (1-\alpha )^2+\lambda ^{\alpha -2}
   \Gamma (2-\alpha )\\
g_{23}&=(2 \alpha -2) \delta ^2 \lambda ^{2 \alpha -3} \Gamma (1-\alpha )^2+(\alpha -2) \delta  \lambda ^{\alpha -3} \Gamma (2-\alpha )\\
g_{31}&=3 \delta ^3 \lambda ^{3 \alpha -3} \Gamma (1-\alpha )^3 \log (\lambda )-3 \delta ^3 \lambda ^{3 \alpha -3} \Gamma (1-\alpha )^3 \psi (1-\alpha )+6 \delta ^2 \lambda ^{2 \alpha -3}
   \Gamma (1-\alpha ) \Gamma (2-\alpha ) \log (\lambda )\\
	&\quad -3 \delta ^2 \lambda ^{2 \alpha -3} \Gamma (1-\alpha ) \Gamma (2-\alpha ) \psi (1-\alpha )-3 \delta ^2 \lambda ^{2 \alpha
   -3} \Gamma (1-\alpha ) \Gamma (2-\alpha ) \psi (2-\alpha )\\
	&\quad+\delta  \lambda ^{\alpha -3} \Gamma (3-\alpha ) \log (\lambda )-\delta  \lambda ^{\alpha -3} \Gamma (3-\alpha ) \psi
   (3-\alpha )\\
g_{32}&=3 \delta ^2 \lambda ^{3 \alpha -3} \Gamma (1-\alpha )^3+6 \delta  \lambda ^{2 \alpha -3} \Gamma (1-\alpha ) \Gamma (2-\alpha )+\lambda ^{\alpha -3} \Gamma
   (3-\alpha )\\
g_{33}&=(3 \alpha -3) \delta ^3 \lambda ^{3 \alpha -4} \Gamma (1-\alpha )^3+3 (2 \alpha -3) \delta ^2 \lambda ^{2 \alpha -4} \Gamma (1-\alpha ) \Gamma (2-\alpha )+(\alpha -3)
   \delta  \lambda ^{\alpha -4} \Gamma (3-\alpha ).
\end{align}
The determinant of this $3\times3$ matrix is given by
\begin{equation}
\delta ^2 \lambda ^{3 \alpha -7} \Gamma (1-\alpha ) \Gamma (2-\alpha ) \Gamma (3-\alpha ) (\psi ^{(0)}(1-\alpha )-2 \psi ^{(0)}(2-\alpha )+\psi ^{(0)}(3-\alpha )).
\end{equation}
A three-dimensional plot with Mathematica shows that the determinant is never zero for $\alpha\in(0,1)$ and $\lambda\in(0,\infty)$. (The factor $\delta^2$ is always positive.) Hence the rank of $G$ is full. This implies local identification of the GMC.

%The Gaussian elimination algorithm yields that the row echelon form has no zero rows. Hence the rank is full. This implies local identification of the GMC.

\item[(b)] The proof works analogously.

\end{enumerate}
\end{proof}

\end{document}